\documentclass{amsart}
\usepackage{amssymb}
\usepackage[all]{xy}
\usepackage{graphicx}
\usepackage[mathcal]{euscript}
\usepackage{verbatim}

\let\oldmarginpar\marginpar
\renewcommand\marginpar[1]{\-\oldmarginpar[\raggedleft\footnotesize #1]%
{\raggedright\footnotesize #1}}

\usepackage[hang]{caption}
\usepackage{mathtools}
\usepackage{manfnt}

\usepackage{color}
\newtheorem{thm}{Theorem}[section]
\newtheorem{lem}[thm]{Lemma}
\newtheorem{prop}[thm]{Proposition}
\newtheorem{cor}[thm]{Corollary}

\newtheorem*{conv}{Convention}
\theoremstyle{definition}
\newtheorem{defn}[thm]{Definition}
\theoremstyle{remark}
\newtheorem{rmk}[thm]{Remark}
\newtheorem{exam}[thm]{Example}

\newcommand{\e}{\epsilon}
\newcommand{\bfe}{\mathbf{e}}

\newcommand{\C}{\mathbb{C}}
\newcommand{\Z}{\mathbb{Z}}
\newcommand{\R}{\mathbb{R}}
\newcommand{\Q}{\mathbb{Q}}
\newcommand{\g}{\gamma}
\renewcommand{\l}{\lambda}

\newcommand{\de}{\delta}
\newcommand{\Ad}{\mathrm{Ad}}
\newcommand{\bAd}{\bar{\Ad}}
\newcommand{\ad}{\mathrm{ad}}
\renewcommand{\a}{\alpha}
\newcommand{\Om}{\Omega}
\newcommand{\fo}{\mathfrak{o}}
\newcommand{\fu}{\mathfrak{u}}

\newcommand{\hfu}{\hat{\fu}}

\renewcommand{\b}{\beta}

\newcommand{\tb}{\tilde{\beta}}
\newcommand{\tbo}{\tilde{\beta}_0}

\newcommand{\fb}{\mathfrak{b}}
\newcommand{\fz}{\mathfrak{z}}
\newcommand{\ddz}{\frac{dz}{z}}
\newcommand{\ddu}{\frac{du}{u}}
\newcommand{\n}{\nabla}
\newcommand{\fg}{\mathfrak{g}}

\newcommand{\fh}{\mathfrak{h}}
\newcommand{\hG}{\hat{G}}
\newcommand{\hfg}{\hat{\fg}}

\newcommand{\hW}{W_{\text{aff}}}
\newcommand{\hatW}{\hat{W}}
\newcommand{\hT}{\hat{T}}
\newcommand{\hU}{\hat{U}}
\newcommand{\hV}{\hat{V}}
\newcommand{\hft}{\hat{\ft}}
\newcommand{\ft}{\mathfrak{t}}

\newcommand{\pd}{d}

\newcommand{\tx}{\tilde{x}}
\newcommand{\ty}{\tilde{y}}

\newcommand{\hN}{\hat{N}}

\newcommand{\ta}{\tilde{a}}
\newcommand{\sG}{\mathcal{G}}
\newcommand{\sQ}{\mathcal{Q}}
\newcommand{\sS}{\mathcal{S}}
\newcommand{\nbr}{[\nabla]}
\newcommand{\cL}{\mathcal{L}}
\newcommand{\fiw}{\frak{i}}

\newcommand{\Opt}{\Psi}

\newcommand{\A}{\mathcal{A}}
\newcommand{\bA}{\bar{\A}}
\newcommand{\bx}{\bar{x}}

\newcommand{\B}{\mathcal{B}}
\newcommand{\bB}{\bar{\B}}

\newcommand{\Ao}{\A_0}
\newcommand{\bAo}{\bA_0}

\newcommand{\bF}{\bar{F}}

\DeclareMathOperator{\Ga}{\mathbb{G}_a}

\DeclareMathOperator{\GL}{\mathrm{GL}}
\DeclareMathOperator{\SL}{\mathrm{SL}}
\DeclareMathOperator{\Sp}{\mathrm{Sp}}

\DeclareMathOperator{\End}{\mathrm{End}}
\DeclareMathOperator{\Aut}{\mathrm{Aut}}
\DeclareMathOperator{\Res}{\mathrm{Res}}
\DeclareMathOperator{\hfz}{\hat{\fz}}

\DeclareMathOperator{\Lie}{\mathrm{Lie}}

\DeclareMathOperator{\Rep}{\mathrm{Rep}}
\DeclareMathOperator{\Spec}{\mathrm{Spec}}
\DeclareMathOperator{\Stab}{\mathrm{Stab}}

\DeclareMathOperator{\gr}{\mathrm{gr}}
\DeclareMathOperator{\Tr}{\mathrm{Tr}}
\DeclareMathOperator{\diag}{diag}
\DeclareMathOperator{\Hom}{Hom}
\DeclareMathOperator{\spa}{span}

\newcommand{\Dx}{\Delta^\times}

\newcommand{\tg}{\hat{\gamma}}

\DeclareMathOperator{\gl}{\mathfrak{gl}}
\DeclareMathOperator{\Gm}{\mathbb{G}_m}

\newcommand{\bfy}{\mathbf{y}}

\newcommand{\bfS}{\mathbf{S}}
\newcommand{\bfr}{\mathbf{r}}

\DeclareMathOperator{\tM}{\widetilde{\mathcal{M}}}

\DeclareMathOperator{\cA}{\mathcal{A}}
\DeclareMathOperator{\Crit}{Crit}

\DeclareMathOperator{\slope}{slope}
\DeclareMathOperator{\Span}{span}
\renewcommand{\P}{\mathbb P}
\newcommand{\fP}{\mathfrak{p}}

\newcommand{\dAo}[1]{d_{\Ao, #1}}

\title[A theory of minimal $K$-types for flat $G$-bundles]{A theory of minimal $K$-types for flat $G$-bundles}

\author{Christopher L.~Bremer}
\address{Department of Mathematics\\
  Louisiana State University\\
  Baton Rouge, LA 70803} 
\email{cbremer@math.lsu.edu} 
\author{Daniel
  S.~Sage} 
\email{sage@math.lsu.edu}

\email{} \thanks{The second author was partially supported by NSF
  grant DMS-1503555 and Simons Foundation Collaboration Grant
  281502).}  \subjclass[2010]{} \keywords{}

\begin{document}
\begin{abstract}

  The theory of minimal $K$-types for $p$-adic reductive groups was
  developed in part to classify irreducible admissible representations
  with wild ramification.  An important observation was that minimal
  $K$-types associated to such representations correspond to
  fundamental strata.  These latter objects are triples $(x, r, \b)$,
  where $x$ is a point in the Bruhat-Tits building of the reductive
  group $G$, $r$ is a nonnegative real number, and $\b$ is a
  semistable functional on the degree $r$ associated graded piece of
  the Moy-Prasad filtration corresponding to $x$.

  Recent work on the wild ramification case of the geometric Langlands
  conjectures suggests that fundamental strata also play a role in the
  geometric setting.  In this paper, we develop a theory of minimal
  $K$-types for formal flat $G$-bundles.  We show that any formal flat
  $G$-bundle contains a fundamental stratum; moreover, all such
  strata have the same rational depth.  We thus obtain a new
  invariant of a flat $G$-bundle called the slope, generalizing the
  classical definition for flat vector bundles.  The slope can also be
  realized as the minimum depth of a stratum contained in the flat
  $G$-bundle, and in the case of positive slope, all such minimal
  depth strata are fundamental.  Finally, we show that a flat
  $G$-bundle is irregular singular if and only if it has positive
  slope.

\end{abstract}
\maketitle
\section{Introduction}

The theory of unrefined minimal $K$-types for representations of
$p$-adic reductive groups arose as a response to two \emph{a priori}
distinct problems concerning admissible representations of $\GL_n$.
First, Bushnell and Fr\"olich~\cite{BuFr} introduced the notion of a
\emph{fundamental stratum} contained in a representation and showed
that for representations containing a fundamental stratum, the stratum
could be used to calculate the constants in the functional equation
of a zeta integral.  Bushnell~\cite{Bus} later proved that any
irreducible admissible representation of $\GL_n$ contains a
fundamental stratum.  In another direction, Howe and
Moy~\cite{Moy89,HM}, motivated by work of Vogan~\cite{V} on
representations of real groups, developed a theory of (unrefined)
minimal $K$-types for $\GL_n$ in order to better understand the
parameterization of irreducible representations.  In particular, they
were interested in establishing a well-defined notion of the
\emph{depth} of a representation $V$ by studying the congruence level
of compact subgroups that fix a vector in $V$.  In retrospect, one
sees that fundamental strata and unrefined minimal $K$-types contain
equivalent information about admissible representations of $\GL_n$
with positive depth.  Thus, fundamental strata can be viewed as a
theory of minimal $K$-types in this case.

This theory plays a vital role in the classification of supercuspidal
representations with wild ramification~\cite{BuKu1,Ku} as well as in
the proof of the local Langlands conjecture for $\GL_n$~\cite{HT,Hen}.
Recent work suggests that minimal $K$-types also play a role in the
geometric setting.  Somewhat unexpectedly, the applications so far
have been on the ``Galois'' side of the correspondence.  Let $G$ be a
complex reductive group with Langlands dual ${}^{L} G$, and let $F$ be
the field of complex Laurent series $\C(\!(z)\!)$.  In the geometric
Langlands program, the role of Galois representations is played by
monodromy data associated to flat ${}^L G$-bundles: over a smooth
complex curve $X$ or the formal punctured disk $\Dx=\Spec(F)$
depending on whether one is in the global or local context.  Consider,
for example, the case $G={}^L G=\GL_n$, so that a flat ${}^L G$-bundle
on $X$ is just a rank $n$ vector bundle endowed with a meromorphic
connection.  For regular singular connections, i.e., those whose
connection matrix at each singular point can be chosen to have simple
poles, the monodromy data is just a representation of the fundamental
group.  Most previous analyses of geometric Langlands have
concentrated on such connections; indeed, a detailed correspondence
has been formulated by Frenkel and Gaitsgory in the ``tame'' case,
where the connections considered are regular singular with unipotent
monodromy~\cite{FrGa}.

Much less is known about the wild case, where irregular
singularities are allowed.  Here, one must also include ``wild''
monodromy data.  This consists of a collection of Stokes matrices at
each singular point, describing the ``jumps'' in the asymptotic
behavior of a horizontal section as it is analytically continued
around each irregular singularity.  Thus, the Stokes data is simply an
enhancement of the usual monodromy data, which allows one to establish
a Riemann-Hilbert correspondence for irregular singular flat vector
bundles~\cite[Chapitre IV]{Mal}.  The deviation of an irregular
singular connection from the regular singular case, or equivalently
the complexity of the Stokes data, is measured by the \emph{slope} of
the connection. By analogy with the $p$-adic case, a geometric theory
of minimal $K$-types ought to detect both whether a flat $G$-bundle is
irregular and the degree of its irregularity, thereby giving a
definition for general $G$ of the discrete invariant slope (akin to
the depth of a representation). It should yield information about the
moduli stack of flat $G$-bundles.  Moreover, it should illuminate
certain transcendental invariants such as the irregular monodromy map,
just as the classical theory was used to calculate constants in the
functional equation of zeta integrals.  Finally, the theory should be
effectively computable, i.e., it should provide an algorithm for
finding a minimal $K$-type associated to a flat $G$-bundle.

The classical approach to studying the local behavior of meromorphic
differential equations in one variable, or equivalently, of
meromorphic connections on $\P^1$, makes use of the naive ``leading
term'' of the connection.  More generally, let $X$ be a smooth curve.
Let $V$ be a rank $n$ vector bundle on $X$ endowed with a meromorphic
connection $\n$, and assume that $y\in X$ is an irregular singular
point.  After choosing a trivialization for $V$ on a formal
neighborhood of $y$ and a local parameter $z$, $\n$ has the local
description
\begin{equation}\label{expansion}
\nabla = d+[\n]=d + (M_{-r}z^{-r}  + M_{-r+1}z^{-r+1}+\dots)\ddz,
\end{equation}
where $[\n]$ (the \emph{matrix} of the connection) is a
$\gl_n(\C(\!(z)\!))$-valued one-form, $M_j \in \gl_n(\C)$, and $r\ge 0$.
When the leading term $M_{-r}$ is well-behaved, asymptotic analysis
using this expansion produces detailed information about the
connection and the form of fundamental solutions at $y$.  (See, for
instance, \cite{Was}.)  For example, if $M_{-r}$ is nonnilpotent, then
the slope of the connection is $r$.  The case when $M_{-r}$ is regular
semisimple has been studied extensively in the literature~\cite{Boa,
  JMU}, and the contribution of the leading term to the Stokes data is
well understood.

This perspective is much less useful when the formal connection has
nonintegral slope because in this situation, the leading term $M_{-r}$
is nilpotent for every formal trivialization at $y$.  This is the case
for many irregular singular connections that arise naturally in the
geometric Langlands program, such as connections corresponding to
cuspidal representations.  For example, Frenkel and Gross have
constructed a rigid flat $G$-bundle on $\P^1$ (for any reductive $G$)
which is the de Rham analogue of the automorphic representation with
the Steinberg representation at $0$ and a certain ``small''
supercuspidal representation at $\infty$~\cite{FrGr}.  Here, the
leading term at the irregular singular point at $\infty$ is always
nilpotent.  More concretely, Witten considers Airy-type connections of
the form
\begin{equation}\label{Airy}
  \nabla = d + \begin{pmatrix} 0 & z^{-s} \\ z^{-s+1} & 0 \end{pmatrix}\frac{dz}{z},
\end{equation}
which always have nilpotent leading terms at $0$~\cite{Wi}.  When
$s=2$, this is the classical Airy connection; when $s=1$, it is the
$\GL_2$ case of the Frenkel-Gross flat $G$-bundle with the roles of
$0$ and $\infty$ reversed.

One obtains the naive notion of the leading term of a connection by
studying the connection in terms of the obvious degree filtration on
$\gl_n(\C(\!(z)\!))$.  In \cite{BrSa1}, the authors introduced a more
powerful notion of the ``leading term'' of a connection by considering
more general filtrations on the loop algebra defined in terms of
\emph{lattice chains}~\cite{BrSa1,Sa00}.  Let $V$ be a
finite-dimensional vector space over $k$.  A lattice chain $\cL =
\{L^j\}_{j \in \Z}$ in $V(\!(z)\!)$ is a collection of lattices (i.e.,
maximal rank $\C[\![z]\!]$-submodules) such that $L^j \supset L^{j+1}$ and
$z L^j = L^{j+e}$ with fixed period $e$.  The stabilizer $P \subset
\GL(V(\!(z)\!))$ of $\cL$ is called a \emph{parahoric subgroup}.  The
lattice chain is uniquely determined by $P$ up to translation of the
indexing of $\cL$; in particular, the period $e=e_P$ is determined by
$P$. One further obtains congruent subalgebras $\fP^m$, consisting of
the endomorphisms that map $L^j$ to $L^{j+m}$ for each $j\in\Z$ in
particular, $\fP = \fP^0$ is the Lie algebra of $P$.  Using this data,
we can associate a triple $(P,r,\b)$ called a \emph{stratum} to the
connection $\n$, where $r$ is a nonnegative integer and $\b$ is a
$\C$-linear functional on $\fP^r/\fP^{r+1}$.  This means that the
matrix $[\n]$ lies in $\fP^{-r}\ddz$ and that the functional on
$\gl(V(\!(z)\!))$ given by $Y \mapsto \Res (\Tr (Y [\n]))$ induces $\b$ on
the quotient algebra.  To give a simple example, the usual leading
term comes from the lattice chain $L^j=z^j V[\![z]\!]$.  The stabilizer
here is the maximal parahoric subgroup $\GL(V[\![z]\!])$ and
$\fP^r=z^r\gl(V[\![z]\!])$.  The connection in \eqref{expansion} satisfies
$[\n](L^j)\subset L^{j-r}\ddz$ while the functional on
$z^r\gl(V[\![z]\!])/z^{r+1}\gl(V[\![z]\!])$ induced by $[\n]$ is the same as
that induced by the leading term $M_{-r}\ddz$.  In the terminology of
\cite{BrSa1}, we say that the stratum
\begin{equation*}\label{naive}(\GL(V[\![z]\!]),r,M_{-r}z^{-r}\ddz)
\end{equation*}
is contained in the induced formal connection at $y$.

As has been noted above, the fact that the stratum \eqref{naive} is
contained in a connection is primarily useful when $M_{-r}$ is
nonnilpotent.  More generally, we say that the stratum $(P,r,\b)$ is
\emph{fundamental} if it satisfies an analogous nondegeneracy
condition.  To be precise, the functional $\b$ comes from some
$Y\ddz\in\fP^{-r}\ddz$ as described above, and the stratum $(P,r,\b)$
is fundamental if the action of $Y$ on the associated graded space
$\bigoplus L^i/L^{i +1}$ is not nilpotent.  It is fundamental strata
that provide a better notion of the leading term of a connection.  In
\cite{BrSa1}, we showed that every formal connection contains a
fundamental stratum.  Moreover, if $(P,r,\b)$ is contained in
$(V(\!(z)\!),\n)$, then $r/e_P\ge\slope(V(\!(z)\!),\n)$ with equality (in the
irregular singular case) precisely when the stratum is fundamental.
We remark that this geometric theory of strata was motivated by the
analogous $p$-adic theory, which plays a crucial role in the
representation theory of $\GL_n$ over $p$-adic
fields~\cite{BuFr,Bus,BuKu1}.

A fundamental stratum contained in a connection at a given point
should provide a coarse approximation to the local behavior of the
connection and its solutions.  This can be seen most strikingly when
the fundamental stratum is \emph{regular}, a condition introduced by
the authors generalizing the classical assumption of regular
semisimplicity of the naive leading term~\cite{BrSa1}.  The data of a
regular stratum includes a (not necessarily split) maximal torus $S$
in $\GL_n(V(\!(z)\!))$. Consider, for example, the collection of
meromorphic connections $(V,\n)$ on $\P^1$ with singularities at
$\bfy=(y_1,\dots,y_m)$ such that the induced formal connection at each
$y_i$ contains an $S_i$-regular stratum of depth $r_i$.  Let $\bfr$
and $\bfS$ denote the collection of $r_i$'s and $S_i$'s respectively.
One can construct the moduli space of such connections as a Poisson
manifold $\tM(\bfy, \bfr)$.  Moreover, the construction of this moduli
space is ``automorphic'', i.e., it is realized as the Poisson
reduction of products of smooth varieties describing local data at the
singularities.\footnote{See \cite{Saisaac} for some explicit examples.}  Finally, the monodromy map induces an integrable
system on $\tM(\bfy,\bfS,\bfr)$, and the authors have analyzed the
relationship between the monodromy foliation and the corresponding
regular strata in~\cite{BrSa2}.

The goal of this paper is to develop a geometric theory of minimal
$K$-types.\footnote{Further extensions and applications appear in
  \cite{BrSa5, KamSa}.}  We will do so by generalizing the theory of
strata and its application to formal flat $G$-bundles for an arbitrary
reductive group $G$ over an algebraically closed field of
characteristic $0$.  Again, one wants to study the behavior of a
connection in terms of suitable filtrations on the loop algebra
$\fg(\!(z)\!)$.  In the general setting, we will use the filtrations
introduced by Moy and Prasad in their theory of minimal $K$-types for
admissible representations of $p$-adic groups~\cite{MP1, MP2}.  Given
any algebraic group $H$ defined over a discrete valuation field, there
is an associated complex called its Bruhat-Tits building.  For any
point $x$ in this building, Moy and Prasad defined a decreasing
$\R_{\ge 0}$-filtration $\{H_{x,r}\}$ on the parahoric subgroup
$H_x=H_{x,0}$ with a discrete number of jumps~\cite{PRag,MP1}.  There
is also a compatible $\R$-filtration on the Lie algebra $\fh$.  An
\emph{unrefined $K$-type} (at least for $r>0$) is then a triple
$(x,r,\b)$ with $\b$ a character of $H_{x,r}/H_{x,r+}$ (with
$H_{x,r+}$ the next step up in the filtration); it is called minimal
if it satisfies a certain nondegeneracy condition.  For $p$-adic
groups, Moy and Prasad showed that every irreducible admissible
representation $V$ of $H$ contains a minimal unrefined $K$-type.
Moreover, if one defines the depth of $V$ to be the smallest $r$
appearing in a $K$-type in $V$, then $(x',r',\b')$ contained in $V$ is
minimal if and only if $r'=r$.  Finally, they showed that two minimal
$K$-types contained in $V$ are closely related; in their terminology,
the minimal $K$-types are \emph{associates} of each
other~\cite{MP1,MP2}.

Returning to the geometric setting, we define a $G(\!(z)\!)$-stratum
to be
a triple $(x,r,\b)$ where $x$ is a point in the Bruhat-Tits building
for $G(\!(z)\!)$, $r\in\R_{\ge 0}$, and $\b$ is a functional on the $r$th
step $\fg(\!(z)\!)_{x,r}/\fg(\!(z)\!)_{x,r+}$ in the filtration on $\fg(\!(z)\!)$
determined by $x$.  Again, we call a stratum fundamental if the
functional $\b$ satisfies a certain nondegeneracy condition.  We
remark that if $r>0$, the exponential map induces an isomorphism
between $\fg(\!(z)\!)_{x,r}/\fg(\!(z)\!)_{x,r+}$ and
$G(\!(z)\!)_{x,r}/G(\!(z)\!)_{x,r+}$, so there is a bijection between strata
and unrefined $K$-types.  In the case of flat vector bundles (i.e.,
$G=\GL_n$), we can recover our previous version of strata by choosing
appropriate points in the building.  Indeed, any parahoric subgroup
$P$ corresponds to a unique facet in the building, and this facet
decomposed as the product of a simplex and $\R$.  If $x$ is any point
lying over the barycenter of this simplex, then the jumps in the
associated filtration occur at $\frac{1}{e_P}\Z$ and
$P^j=G_{x,j/e_P}$, where $e_P$ is the period of a lattice chain with
stabilizer $P$.

Our main results, given in Theorems~\ref{MP},~\ref{associateness}, and
\ref{slopethm}, are the geometric analogue of Moy and Prasad's theorem
on minimal $K$-types for admissible representations of $p$-adic
groups.  We show that any formal flat $G$-bundle $(\sG,\n)$ contains a
fundamental stratum $(x,r,\b)$ with $r$ a nonnegative rational number
and all such strata have the same depth.  In fact, we exhibit an
explicit algorithm for finding such a fundamental stratum. Moreover,
any two fundamental strata contained in $\n$ are associates of each
other.  We thus obtain a new invariant of a flat $G$-bundle called the
slope, generalizing the classical definition for flat connections.  We
further show that $(\sG,\n)$ is irregular singular if and only if
$\slope(\n)>0$.  Finally, we prove that a stratum $(x',r',\b')$
contained in the formal $G$-bundle $\n$ satisfies $r'\ge \slope(\n)$,
and in the irregular singular case, it is fundamental if and only if
$r'=r$.

We remark that there are other approaches to defining the slope.
In~\cite{FrGr}, Frenkel and Gross suggest a procedure which involves
pulling the flat $G$-bundle back to a ramified cover where the
connection matrix (for a suitable trivialization) is well-behaved with
respect to the usual degree filtration.  Corollary~\ref{FGslope} shows
that their approach is well-defined and gives the same invariant as
our definition.  There is also a recent preprint of Chen and
Kamgarpour which shows how to define slope in terms of
opers~\cite{ChKa}.  These alternate definitions are discussed in more
detail in Remark~\ref{sloperemark}.

We would like to thank D. Gaitsgory and M. Kamgarpour for helpful
conversations.  We would also like to thank one of the referees for
suggestions regarding a trivialization-free approach to stratum
containment.

\section{Preliminaries and main results}\label{Prelim}

In this section, we state the main results of the paper.  In order to
make the paper accessible to readers with a background in $p$-adic
groups, who may not be familiar with the geometric setting, we begin
by illustrating the theory for flat vector bundles.  We also provide
some background on flat $G$-bundles.

Throughout the paper, $k$ will be an algebraically closed field
of characteristic $0$, and $F=k(\!(z)\!)$ will be the field of formal Laurent
series over $k$ with ring of integers $\fo=k[\![z]\!]$.  We write $\Dx =
\Spec (F)$ for the formal punctured disk and $\Om^1 = \Om^1_{F/k}$ for
the space of differential one-forms on $\Dx$.  We will denote the
Euler vector field $z\frac{\pd}{\pd z}$ by $\tau$.  Finally,
$\iota_\tau$ will be the inner derivation by $\tau$, so that a
$1$-form $\omega$ can be written $\omega=\iota_\tau(\omega)\ddz$.

\subsection{Flat vector bundles}\label{flatvb}

A flat vector bundle over $\Dx$ is an $F$-vector space $U$ equipped
with a connection $\n$, which is a $k$-derivation $\n : U \to \Om^1
(U)$.  Throughout, we shall assume that $U$ is finite dimensional.
The simplest example is the trivial connection on $ F^n$. In the
standard basis $\{e_i \}_{i = 1}^n$, $\n$ is the usual exterior
derivative $d$ defined by $d(e_i) = 0$ and extended to $F^n$ by
$k$-linearity and the Leibniz rule: for $f \in F$, $d (f u) =
\frac{df}{dz} u + f d(u)$.  More generally, for any $n \times n$
matrix $M$ with coefficients in $\Om^1$, it is easily checked that
the following is a $k$-derivation from $F^n$ to $\Om^1(F^n)$:
\begin{equation*}
\n (u) = d(u) + M u.
\end{equation*}
We will use the shorthand $\n = d + M$ for this operator.  Moreover,
any connection on $F^n$ is of this form.  Accordingly, if $(U,\n)$ is
a rank $n$ connection and $\phi : U \to F^n$ is a trivialization, then
there exists a matrix $[\n]_\phi$ of $1$-forms such that the induced
connection $\n_\phi$ on $F^n$ can be written as $\n_\phi = d +
[\n]_\phi$.  One can also write $\n$ in terms of an ordinary matrix as
$\n_\phi=d+\iota_\tau([\n]_\phi)\ddz$.  The left action of $\GL_n(F)$
on trivializations has the following effect on the matrix of a
connection, known as a gauge transformation: for $g \in \GL_n(F)$,
\begin{equation}\label{gaugevb}
\nbr_{g \phi} = g \cdot \nbr_\phi = g\nbr_\phi g^{-1} - (dg)
g^{-1}.
\end{equation}

We remark that if $T\in\Om^1(\End(U))$, the formula $d+T$ does not
give a well-defined connection on $U$.  However, if $U$ is endowed
with a fixed $k$-structure, say $U=U_k\otimes F$, then $d+T$ makes
sense: $(d+T)(fv)=\frac{df}{dz} v+fT(v)$ for $f\in F$ and $v\in U_k$.
In particular, below we will frequently consider flat vector bundles
with $U=\hV:=V\otimes F$, where $V$ is a $k$-vector space, and
expressions of the form $d+T$ will be defined in terms of the natural
$k$-structure.

A flat vector bundle $(U,\n)$ is called \emph{regular singular} if
there exists an $\fo$-lattice $L \subset U$ with the property that
$\nabla(L) \subset L \otimes_{\fo} \Omega^1_{\fo/k} (1)$.
Equivalently, $(U,\n)$ is regular singular if and only if there exists
a lattice $L$ for which $(\iota_\tau\circ \n) (L) \subset L$.
This means that there is some trivialization $\phi$ with
respect to which the matrix $[\n]_\phi$ has at worst a simple pole.
Otherwise, $\nabla$ is said to be \emph{irregular singular}.

The deviation of an irregular singular connection from the regular
singular case is measured by an invariant called the \emph{slope}.  We
recall the precise definition.  Fix a lattice $L \subset U$.  If $\bfe
= \{e_j\}$ is a finite collection of vectors in $U$, we define
$v(\bfe) = m$ if $m$ is the greatest integer such that $\bfe \subset
z^m L$.   Let $(U, \nabla)$ be an irregular flat
vector bundle. Take $\bfe$ to be any basis for $U$.  By a theorem of Katz
\cite[Theorem~II.1.9]{De}, there is a unique positive rational number
$r$ such that the subset of $\Q$ given by
\begin{equation}\label{irregprecise}
\{ v(\iota_\tau\circ\n)^i \bfe) + r i \mid i > 0\}
\end{equation}
is bounded.  Here, $(\iota_\tau\circ\n)^i \bfe = \{
(\iota_\tau\circ\n)^i (e_j)\}$.  This number is independent of the
choice of basis $\bfe$ and is called the slope of $(U,\n)$.  For a
regular singular connection, the set in \eqref{irregprecise} is never
bounded for positive $r$, and we say that its slope is
$0$.

\subsection{Filtrations on $F$-vector spaces}\label{filtF}

A previous paper by the authors~\cite{BrSa1} is concerned with the
interaction between flat vector bundles $(U, \n)$ and decreasing
$\R$-indexed filtrations on $U$ consisting of $\fo$-lattices $\{U_r
\}_{r \in \R}$.  We will only consider \emph{periodic} (more
specifically, $1$-periodic) filtrations, for which $z U_r = U_{r +1}$
for all $r$.  Such filtrations are parameterized by points in a
certain complex $\B=\B(\GL(U))$ called the Bruhat-Tits building for
$\GL(U)$.  We will not describe this correspondence in detail here;
see Section~\ref{mpfilt2} where we discuss analogous filtrations for
an arbitrary reductive group.  For the present, we only recall a few
properties of the building.\footnote{References for Bruhat-Tits
  buildings include the survey article \cite{Ti} and the book
  \cite{Land}, as well as the original papers of Bruhat and
  Tits~\cite{BT1,BT2}.}  First, $\B$ is the union of $n$-dimensional
real affine spaces called \emph{apartments}, which are in one-to-one
correspondence with the set of split maximal tori in $\GL(U)$.  A
choice of basis for $U$ determines a specific apartment together with
an origin for the affine space.  Second, $\B$ is endowed with a
surjection onto a simplicial complex $\bB$ called the reduced building
with fibers isomorphic to $\R$; furthermore, one can decompose the
building as $\B=\bB\times\R$.  If $x\in\B$ lies over $\bx\in\bB$, then
the filtrations for points lying above $\bx$ are precisely those of
the form $U'_{r}=U_{x,r+s}$ for some fixed $s\in\R$.  Finally, the
origin determined by a choice of basis lies over a fixed vertex in
$\bB$.

\begin{exam}[Degree filtration]\label{degree}
  Let $U = F^n$ and $U_j = z^j \fo^n$ whenever $j \in \Z$.  This
  $\Z$-filtration extends to an $\R$-filtration by $U_i = U_{\lceil i
    \rceil}$. This periodic filtration corresponds to the origin of
  $\B$ determined by the standard basis of $F^n$.
\end{exam}

The most familiar class of periodic filtrations comes from
\emph{lattice chains} in $U$.  Recall that a lattice chain $\cL =
\{L^j\}_{j \in \Z}$ in $U$ is a collection of $\fo$-lattices in $U$
such that $L^j \supset L^{j+1}$ and $z L^j = L^{j+e}$ with fixed
period $e$~\cite{BrSa1,Sa00}. The corresponding $\R$-filtration is
given by setting $U^{\cL}_{j/e}=L^{\lceil j\rceil}$.  Note that
lattice chain filtrations are uniform in the sense that the ``jumps''
in the filtration are evenly spaced a distance $1/e$ apart.  More
formally, given any filtration $\{U_r\}$, set $U_{r+}=\cup_{s>r}U_s$.
The critical numbers of the filtration are those $r$ for which $U_r\ne
U_{r+}$.  A filtration is called \emph{uniform} if the distance
between consecutive critical numbers is a constant.  An arbitrary
uniform filtration is obtained from an appropriate $\{U^{\cL}_r\}$ by
translating the indices by a constant.

Uniform filtrations correspond to distinguished points in the
building.  Let $P \subset \GL(U)$ be the stabilizer of $\cL$;
equivalently, it is the stabilizer of any uniform filtration coming
from $\cL$ up to translation of the indices.  Such subgroups are
called \emph{parahoric subgroups}, and they parameterize the facets in
both $\B$ and $\bB$ in a way compatible with the natural map
$\B\to\bB$.  The uniform filtrations coming from $\cL$ are precisely
those $U_{x,r}$ with $x\in\B$ lying above the barycenter of the
simplex in $\bB$ corresponding to $P$ (c.f.
Appendix~\ref{glappendix}).  We denote the Lie algebra of $P$ by
$\fP$; it is called a \emph{parahoric subalgebra}.

A periodic filtration on $U$ also induces a periodic filtration on the
endomorphism ring $\gl(U)$ of $U$.  Indeed, if $x\in\B$ corresponds to
the filtration $U_{x,r}$, then one sets $\gl(U)_{x,r}=\{X\in\gl(U)\mid
X(U_{x,s})\subset U_{x,s+r}\text{ for all } s\}$.  In particular, if
$x$ is in the facet corresponding to the parahoric subgroup $P$, then
$\gl(U)_{x,0}=\fP$.  These subspaces should be viewed as
generalized congruence subalgebras.  To see why, let $\cL$ be a
lattice chain with period $e$.  The corresponding
sequence of congruence subalgebras is given by $\fP^m
=\{X\in\gl(U)\mid X(L^i)\subset L^{i+m}\text{ for all } i\}$, with
$\fP^0 = \fP$. It is now easy to check that the filtration
$\{\gl(U)^{\cL}\}$ with critical numbers $\frac{1}{e}\Z$ and
$\gl(U)^\cL_{m/e}=\fP^m$ for all $m$ is the same as the filtration
induced from $\{U^{\cL}_r\}$.

The degree filtration is an example of a uniform filtration coming
from a lattice chain of period $1$.  The corresponding parahoric
subgroup is the maximal parahoric subgroup $\GL_n(\fo)$.  We now give
an example corresponding to a minimal parahoric (or Iwahori) subgroup.

\begin{exam}[Iwahori filtration]\label{Iwahori}
  Let $U = F^n$.  For $m\in\Z$, write $m=sn+j$ with $0\le j<n$, and
  let $L^m$ be the lattice with $\fo$-basis $\{z^s e_i\mid i\le
  n-j\}\cup\{z^{s+1} e_i\mid i>n-j\}$.  The lattice chain $\cL=\{L^m\}$
  has period $n$, so that the associated periodic filtration has
  critical numbers $\frac{1}{n}\Z$.  The stabilizer of this lattice
  chain is the standard Iwahori subgroup $I$, consisting of the
  pullback of the standard upper triangular Borel subgroup under the
  natural map $\GL_n(\fo)\to\GL_n(k)$.  The point $x_I\in\B$
  corresponding to the Iwahori filtration lies above the barycenter of
  the chamber (i.e., maximal facet) in $\bB$ corresponding to $I$.
\end{exam}

We will also need to consider periodic filtrations on the space of
$1$-forms $\Om^1(U)$ and on the smooth $k$-dual space $\gl(U)^\vee$, i.e., the space of
$k$-functionals vanishing on a nonempty bounded open subgroup.  First,
we define $\Om^1(U)_{x,r}$ as the image of $U_{x,r}$ via the
 isomorphism $U\to\Om^1(U)$ given by $u\mapsto u\ddz$.    Next, observe that there
is a pairing
\begin{equation}\label{ResTr}\gl(U)\times \gl(U)\to k,  \qquad (X,Y)\mapsto
\Res\Tr(XY)\tfrac{dz}{z},
\end{equation}
which induces an isomorphism $\gl(U)\to \gl(U)^\vee$.  We set
$\gl(U)^\vee_{x,r}$ to be the image of $\gl(U)_{x,r}$ under this
isomorphism.

\subsection{$U$-strata and flat vector bundles}
In order to relate filtrations to connections, we introduce the notion
of a $U$-stratum.  We will denote the $k$-dual of a finite-dimensional
$k$-vector space $W$ by $W^\vee$.

\begin{defn}\label{ustrat} A \emph{$U$-stratum} of depth $r$ is a triple $(x,r,\b)$,
  where $x\in\B$, $r\ge 0$, and
  $\b\in(\gl(U)_{x,r}/\gl(U)_{x,r+})^\vee$.
\end{defn}

The pairing \eqref{ResTr} induces an isomorphism
$\left(\gl(U)_{x,r}/\gl(U)_{x,r+}\right)^\vee\cong\gl(U)^\vee_{x,-r}/\gl(U)^\vee_{x,-r+}$,
and we say that $\tb\in\gl(U)^\vee_{x,-r}$ is a \emph{representative} of
$\b$ if $\tb+\gl(U)^\vee_{x,-r+}$ corresponds to $\b$.  The stratum
$(x,r,\b)$ is called \emph{fundamental} if every representative is
nonnilpotent.  (We define nilpotent elements in $\gl(U)^\vee$ via
transport of structure from $\gl(U)$.)

We can now show how $U$-strata can be used to define the leading term
of a flat connection with respect to a periodic filtration.  Given
$\tb\in\gl(U)^\vee$, set $X_{\tb}\in\Om^1(\gl(U))$ equal to the unique
$1$-form such that $\Res\tb(Y)\ddz= \Res\Tr(X_{\tb}
Y)$. 

\begin{defn}\label{vbcontaindef}
  Let $(U, \n)$ be a flat vector bundle.  We say that $(U, \n)$
  contains the stratum $(x, r, \b)$ if, given any (or equivalently,
  some) representative $\tb$ for $\b$,
\begin{equation}\label{vbcontain}
\begin{aligned}
(\n - j\ddz - X_{\tb})( U_{x, j}) & \subseteq \Om^1(U)_{x, (j-r)+}, & \forall j \in \R.
\end{aligned}
\end{equation}
\end{defn}

\begin{exam} Let $o\in\B$ correspond to the degree filtration on $F^n$
  determined by the standard basis as in Example~\ref{degree}.  The
  connection given in \eqref{expansion} makes $F^n$ into a flat vector
  bundle that contains the stratum $(o,r,M_{-r}z^{-r}\frac{dz}{z})$.
  Here, $M_{-r}\frac{dz}{z}$ is viewed as the functional on
  $z^r\gl_n(k)\cong z^r\gl_n(\fo)/z^{r+1}\gl_n(\fo)$ given by the
  residue of the trace form.  This stratum is fundamental if and only
  if $M_{-r}$ is nonnilpotent.
\end{exam}

\begin{exam} The connection on $F^2$ given by \eqref{Airy} contains
  the nonfundamental stratum $(o,s,\left(\begin{smallmatrix}
      0&z^{-s}\\0&0
    \end{smallmatrix}\right)\ddz)$ based at the degree filtration.
    However, it contains a fundamental stratum with respect to the
    Iwahori filtration of Example~\ref{Iwahori}, namely
\begin{equation*}(x_I,s-\frac{1}{2},\left(\begin{smallmatrix} 0&z^{-s}\\z^{-s+1}&0
    \end{smallmatrix}\right)\ddz).
\end{equation*}
Setting $\frak{i}=\Lie(I)$, the $s$th Iwahori congruence
subalgebra is given by \begin{equation*} \fiw^s=\begin{pmatrix}z^{\lceil
      s/2\rceil}\fo &z^{\lfloor s/2\rfloor}\fo\\z^{\lfloor s/2\rfloor
      +1}\fo&z^{\lceil s/2\rceil}\fo
  \end{pmatrix}.
\end{equation*}  The $1$-form in the stratum acts by $\Res\Tr$ on
$\Span_k\{\left(\begin{smallmatrix} 0&0\\ z^{s}&0
    \end{smallmatrix}\right),\left(\begin{smallmatrix}
      0&z^{s-1}\\0&0
    \end{smallmatrix}\right)\}\cong
  \frak{i}^{2s-1}/\frak{i}^{2s}$.
  \end{exam}

\begin{rmk}\label{vbstratarem} The theory of $U$-strata described here
  is somewhat different from that given in~\cite{BrSa1}.  For example,
  the definition of strata in \cite[Definition 2.13]{BrSa1} (as well
  as the original definition from \cite{BuFr} in the context of
  $p$-adic representation theory) only allows for filtrations coming
  from lattice chains.  Also, the present definition of containment of
  a stratum in a connection differs from that in \cite[Definition
  4.1]{BrSa1}; it is equivalent for strata of positive depth.  This is
  discussed in the appendix; see, in particular,
  Proposition~\ref{vbequiv}.
\end{rmk}

\subsection{Flat $G$-bundles}\label{flatg}

We now turn from flat vector bundles to flat $G$-bundles.  We will
need to fix some notation.  

Let $G$ be a connected reductive group over $k$ with Lie algebra
$\fg$.  The category of finite-dimensional representations of $G$ over
$k$ will be denoted by $\Rep(G)$. We fix a nondegenerate invariant
symmetric bilinear form $\left<, \right>$ on $\fg$ throughout.  We set
$\hG=G_F$ and $\hfg=\fg\otimes_k F$; note that $\hG$ represents the
functor sending a $k$-algebra $R$ to $G(R(\!(z)\!))$.  We will use the
analogous notation $\hat{H}$ and $\hat{\fh}$ for any algebraic group
$H$ over $k$.

A formal principal $G$-bundle $\sG$ is a principal $G$-bundle over
$\Dx$.  The $G$-bundle $\sG$ induces a tensor functor from $\Rep(G)$
to the category of formal vector bundles via
$(V,\rho)\mapsto V_\sG = \sG \times_G V$, and by Tannakian formalism,
this tensor functor uniquely determines $\sG$.  Formal principal
$G$-bundles are trivializable, so we may always choose a
trivialization $\phi : \sG \to \hG$.  Note that this trivialization
induces an isomorphism between the groups $\Aut(\sG)$ and
$\hG$. Moreover, there is a left action of
$\hG$ on the set of trivializations of $\sG$.  The trivialization
$\phi$ induces a compatible collection of maps
$\phi_V:V_\sG\to\hV:=V\otimes_k F$; we will usually omit the
superscript from the notation.

A flat structure on a principal $G$-bundle is a formal derivation $\n$
that determines a compatible family of flat connections $\n_V$ (which
we usually write simply as $\n$) on $V_\sG$ for all
$(V,\rho) \in \Rep(G)$.  In practice, once one has fixed a
trivialization $\phi$ for $\sG$, $\n$ may be expressed in terms of a
one-form with coefficients in $\hfg$.  This means that there exists an
element $\nbr_\phi \in \Om^1(\hfg)$, called the \emph{matrix} of
$\n$ with respect to the trivialization $\phi$, for which the induced
connection on $\hV$ is given by $d+\rho(\nbr_\phi)$.  We will formally
write $\n_\phi=d+\nbr_\phi$ for the flat structure on $\hG$ induced by
$\phi$.  To express the effect of change of trivialization (or gauge
change) on the matrix, we first observe that there is a natural action
of $\hG$ on $\Om^1(\hfg)$ which we will denote by
$\Ad^*$.\footnote{As we will see in Remark~\ref{coadjoint}, this may be viewed as
  the coadjoint action on $\fg^\vee\otimes F$.}  Recalling that
$\iota_\tau$ is the inner derivation by the Euler vector field, we can
write $[\n]_\phi=\iota_\tau([\n]_\phi)\ddz\in\hfg\ddz$; with this notation,
$\Ad^*(g) (\nbr_\phi)=\Ad(g)(\iota_\tau([\n]_\phi))\ddz$.  The gauge
transformation action of $\hG$ is then given by 
\begin{equation}\label{gauge}
\nbr_{g \phi} = g \cdot \nbr_\phi = \Ad^*(g) (\nbr_\phi) - (dg) g^{-1}.
\end{equation}
Here, the right-invariant Maurer-Cartan form $(dg) g^{-1}$ lies in
$\Om^1(\hfg)$ and may be calculated explicitly.

We remark that there is an equivalence of categories between flat
$\GL_n$-bundles and flat rank $n$ vector bundles given by
$(\sG,\n)\mapsto(V_\sG,\n_V)$, where $V$ is the standard
representation.

\subsection{The Bruhat-Tits building}\label{btbuild}

As for flat vector bundles, we will study flat $G$-bundles in terms of
appropriate periodic $\R$-filtrations; however, here we will need
compatible filtrations on $\hV$ for all $V\in\Rep(G)$.
We will consider a class of filtrations introduced by Moy and Prasad
that are parameterized by a complex $\B(\hG)$ called the Bruhat-Tits
building of $\hG$~\cite{MP1}. 

In this section, we recall some basic information about the Bruhat-Tits
building.  Fix a maximal torus 
$T\subset G$ with corresponding Cartan subalgebra
$\ft$.  Let $N=N(T)$ be the normalizer of $T$, so that the
Weyl group $W$ of $G$ is isomorphic to $N/T$.  We denote the set of
roots with respect to $T$ by $\Phi$; if
$\alpha \in \Phi$, $U_\a \subset G$ is the associated root subgroup
and $\fu_\a \subset \fg$ is the weight space for $\ft$ corresponding
to $\a$. We will write $Z$ for the center of $G$ and $\fz$ for its Lie
algebra.  We also let $\hW=N(T(F))/T(\fo)$ denote the affine Weyl
group.

We will define two versions of the building, the reduced building
$\bB(\hG)$ and the enlarged building $\B(\hG)$, which we will simply
refer to as the building of $\hG$.  Both are defined as appropriate
unions of affine spaces called apartments, which are in one-to-one
correspondence with the split maximal tori in $\hG$. We start by
defining the apartments $\bAo=\bA(\hT) $ and $\Ao=\A(\hT)$ associated
to $\hT$, which we call standard apartments; they are affine spaces
isomorphic to $X_*(T\cap[G,G])\otimes_\Z \R$ and $\bAo\times \left(X_*
  (Z) \otimes_\Z \R\right)\cong X_*(T) \otimes_\Z \R$ respectively.
The standard apartments are endowed with a cell structure induced by
the roots.  Explicitly, the facets are intersections of half-spaces
determined by the affine hyperplanes $\{x\in\Ao\mid \a(x)=j\}$ for
$\a\in\Phi$ and $j\in\Z$ (and similarly for $\bAo$).  A facet in
$\bAo$ is a polysimplex while the facets in $\Ao$ are pullbacks of
those in $\bAo$ under the projection map. The group $N(\hT)$ acts on
these apartments by affine transformations which preserve the cell
structure.

We next need to define the parahoric subgroup $\hG_x$ associated to
$x\in\Ao$.  Given $\a\in\Phi$, let $\hU_{\a,x}$ be the image of
$z^{\lceil -\a(x)\rceil}\fo$ under the isomorphism $F\cong\hU_\a$.
The subgroup $\hG_{x}$ is then generated by $T(\fo)$ and the $\hU_{\a,
  x}$'s.  Its pro-unipotent radical is denoted by $\hG_{x,+}$.  The
reduced building $\bB(\hG)$ is defined as the quotient of
$\hG\times\bAo$ by the following equivalence relation \cite[\S
9]{Land}:
\begin{equation*} (g,x)\sim (h,y) \iff \text{ there exists $n\in
    N(\hT)$ such that $y=n\cdot x$ and $ g^{-1}hn\in\hG_{x+}$.}
\end{equation*}
We then set $\B(\hG)=\bB(\hG)\times \left(X_* (Z) \otimes_\Z
  \R\right)$.  We will view $\bB(\hG)$ as a subset of $\B(\hG)$ via
the zero section.  The group $\hG$ acts on the buildings, and the
buildings inherit cell structures from the standard apartments via
these actions.  The cell structures are compatible with the projection
$\B(\hG)\to\bB(\hG)$.

Since every point is conjugate to an element in the standard
apartments, the buildings are covered by translates of the standard
apartments.  These are the apartments of the buildings.  They are
parameterized by the split maximal tori in $\hG$.  More specifically,
any split maximal torus $T'$ is of the form $g\hT g^{-1}$ for some
$g\in\hG$; we define the corresponding apartments via $\cA(T')=g\Ao$
and $\bA(T')=g\bAo$.

Given $x\in \B(\hG)$, we define the parahoric subgroup $\hG_x$ to be
the connected stabilizer of $x$; it is a finite index subgroup of
$\Stab(x)$.\footnote{If $x\in\Ao$, this definition agrees with our
  previous definition.  Indeed, this follows from the fact that for
  $x\in\Ao$, $\Stab(x)$ is generated by $\Stab_{N(\hT)}(x)$ and the
  $\hU_{\a, x}$'s~\cite[\S 8]{Land}.}  The corresponding Lie
subalgebra is denoted $\hfg_x$.  The parahoric subgroups are in
bijective correspondence with the facets of $\B(\hG)$.  Indeed,
$\hG_x=\hG_y$ if and only if $x$ and $y$ are in the same facet.  In
particular, the parahoric subgroups only depend on the image of $x$ in
$\bB(\hG)$.  The maximal parahoric subgroup $G(\fo)$ corresponds to
the origin in $X_*(T\cap[G,G])\otimes_\Z \R$.  See
\cite{BT1,BT2,Ti,Land} for more details.

We will also need to consider the Bruhat-Tits building $\B(\Aut(\sG))$
for the $F$-group of automorphisms of the formal principal $G$-bundle
$\sG$.  This may
be defined in exactly the same way, starting with a fixed
maximal torus of $\Aut(\sG)$. Alternatively, any trivialization of $\sG$
induces an isomorphism between $\B(\Aut(\sG))$ and $\B(\hG)$, and in
fact, we see that $\B(\Aut(\sG))=\B(\hG)\times^{\hG}\sG$.

\subsection{Moy-Prasad filtrations}\label{mpfilt2}

We now define filtrations associated to points in $\B(\hG)$.

Let $V$ be a finite-dimensional representation of $G$ over $k$, and
let $\hV$ be the corresponding representation of
$\hG$.  For any $x\in \B(\hG)$, the Moy-Prasad filtration associated to $x$
is a decreasing $\R$-filtration $\{\hV_{x, r} \mid r \in \R\}$ of
$\hV$ by $\fo$-lattices~\cite{GKM06,MP1}.  We briefly review the
construction.  First, assume $x\in\Ao$.  In this case, the filtration
can be obtained from an $\R$-grading on $V\otimes_k k[z,z^{-1}]$ (cf.
~\cite{GKM06}).  If $\chi \in X^*(T)$, let $V_\chi \subset V$ be the
weight space $\{v \in V \mid s v = \chi(s) v \quad \forall s \in T
\}$.  The $r$th graded subspace is defined to be
\begin{equation}\label{graded}
\hV_{x,\Ao} (r) = \bigoplus_{\chi (x) + m = r} V_\chi z^m \subset \hV.
\end{equation}
Note that the set $\{r\mid \hV_{x,\Ao} (r)\ne 0\}$ is discrete
and
closed under translations by $\Z$.  For any $r \in \R$, define
\begin{equation*}
\begin{aligned}
\hV_{x, r} & = \prod_{s \ge r} \hV_{x,\Ao}(s)\subset \hV; &
\hV_{x, r+} & = \prod_{s > r} \hV_{x,\Ao} (s) \subset \hV.
\end{aligned}
\end{equation*}
If $x\in\B(\hG)$ is arbitrary, we write $x=gy$ for $g\in\hG$, $y\in\Ao$,
and set $\hV_{x,r}=g\hV_{y,r}$.  It is shown in \cite{MP1} that this
definition is independent of the choice of $g$ and $y$.

The collection of lattices $\{\hV_{x, r}\}$ is the Moy-Prasad
filtration on $\hV$ associated to $x$.  It is immediate that
$\hV_{x,r+1}=z\hV_{x,r}$.  Note that if $x\in\Ao$, $\hV_{x, r} /
\hV_{x, r+}\cong \hV_{x,\Ao}(r) \ne \{0\}$ if and only if there exists
$\chi \in X^*(T)$ such that $V_\chi \ne \{0\}$ and $r - \chi (x) \in
\Z$.  We call the real numbers for which $\hV_{x,r}\ne\hV_{x,r+}$ the
\emph{critical numbers} of $V$ at $x$, and denote the set of such
points by $\Crit_x(V)$.  We will write $\Crit_x$ for $\Crit_x(\fg)$,
the critical numbers of the adjoint representation.  The set
$\Crit_x(V)$ is a discrete subset of $\R$ closed under translation by
$\Z$.  If the set of weights in $V$ is closed under inversion, then
$\Crit_x(V)$ is also symmetric around $0$.  In particular, this is the
case for representations on which $\fz$ acts trivially such as the
adjoint and coadjoint representations $\fg$ and $\fg^\vee$.  We further
observe that when $\fz$ acts trivially, the Moy-Prasad filtrations
depend only on the image $\bx\in\bB(\hG)$ of $x$.

Moy-Prasad filtrations are canonically associated to points in the
building, but this is not true for gradings.  Given $T'\subset G$ a
maximal torus and $x'\in\A'=\A(\hT')$, one can define an analogous
grading $\hV_{x',\A'}(r)$; moreover, if $T'=gTg^{-1}$ and $gx=x'$ for
some $g\in G$, then $\hV_{x',\A'}(r)=g\hV_{x,\Ao}(r)$.  Accordingly, if
$gx=x$, but $gTg^{-1}\ne T$, we need not have
$\hV_{x,\A'}(r)=\hV_{x,\Ao}(r)$, so the grading depends on a choice of
apartment containing $x$.  In this paper, we will only consider
gradings coming from $\Ao$, so we will usually simplify notation by
writing $\hV_x(r)$ instead of $\hV_{x,\Ao}(r)$.

We will also need to consider Moy-Prasad filtrations on the space of
$1$-forms $\Om^1(\hV)$ and on the smooth $k$-dual space $\hfg^\vee$.
We define $\Om^1(\hV)_{x,r}$ to be the image of $\hV_{x,r}$ via the
 isomorphism $\hV\to\Om^1(\hV)$ given by $u\mapsto u\ddz$.    Next,
 our fixed  nondegenerate invariant symmetric bilinear form $\left<,
 \right>$ on $\fg$ induces an isomorphism $\hfg\to\hfg^\vee$ given by
 $X\mapsto \Res\left<X,\cdot\right>\ddz$.  We set
$\hfg^\vee_{x,r}$ equal to the image of $\hfg_{x,r}$ under this
isomorphism.  Note that composing the isomorphisms
$\Om^1(\hfg)\to\hfg$ and $\hfg\to\hfg^\vee$ gives an isomorphism
$\Om^1(\hfg)\to\hfg^\vee$ which preserves Moy-Prasad filtrations.

There is also a corresponding filtration
$\{\hG_{x,r}\}_{r\in\R_{\ge 0}}$ of the parahoric subgroup
$\hG_x=\hG_{x,0}$ for $x\in\B(\hG)$~\cite[Section 2.6]{MP1}.  To define
$\hG_{x,r}$, first let $\hT_r$ be the subgroup of $\hT$ generated by
the images of the $\lceil r\rceil$th congruent subgroup in
$\fo^\times$ under each cocharacter $\lambda\in X_*(T)$.  Similarly,
if $\psi\in\Phi$, let $\hU_{\psi, x, r}$ be the image of
$z^{\lceil r-\psi(x)\rceil}\fo$ under the isomorphism
$F\cong\hU_\psi$.  The subgroup $\hG_{x,r}$ is generated by the
$\hT_r$ and the $\hU_{\psi, x, r}$'s.  We set
$\hG_{x, r+}=\bigcup_{s>r}\hG_{x,s}$.  In particular, we write
$\hG_{x+}=\hG_{x,0+}$; it is the pro-unipotent radical of
$\hG_x$~\cite[p.397]{MP1}.  (We use similar notation for the Lie
algebras: $\hfg_{x}=\hfg_{x,0}$ and $\hfg_{x+}=\hfg_{x,0+}$.)  For
$r > 0$, there is a natural isomorphism
$\hG_{x, r} / \hG_{x, r+} \cong \hfg_{x, r} / \hfg_{x,
  r+}$~\cite[p.399]{MP1},
so that this quotient group is unipotent.  However,
$\hG_{x} / \hG_{x+}$ is reductive.  Indeed, in Section~\ref{subfilt},
we will construct an explicit isomorphism between this group and a
maximal rank reductive subgroup of $G$. The parahoric subgroup
$\hG_{x}$ stabilizes $\hV_{x, r}$, and $\hV_{x, r}/\hV_{x,r+}$ is a
representation of $\hG_{x} / \hG_{x+}$.  Note that when $G$ is a
torus, i.e., $G=T$, then the reduced building is a point.
Accordingly, there is a unique Moy-Prasad filtration on $\hT_0=T(\fo)$
and $\hft$; the filtration on $T(\fo)$ is given by the subgroups
$\hT_r$ above.

One can similarly define Moy-Prasad filtrations on the vector bundles
$V_\sG$ parameterized by points in $\B(\Aut(\sG))$.  Of course, fixing
a trivialization identifies these filtrations with filtrations on the
$F$-vector spaces $\hV$.  To be explicit, let $\phi$ be a
trivialization of $\sG$.   Given $x\in \B(\Aut(\sG))$, let $\phi(x)$
be the induced point in $\B(\hG)$.  Then, the trivialization $\phi_V$
identifies $(V_\sG)_{x,r}$ and $\hV_{\phi(x),r}$ for all $r$.

\subsection{Strata and flat $G$-bundles}\label{s:gstrata}

In this section, we generalize the geometric theory of strata for
$\GL_n$ given in \cite{BrSa1} to arbitrary reductive groups.  More
precisely, given a formal flat $G$-bundle $(\sG,\n)$, its leading term
with respect to a Moy-Prasad filtration is described either in terms
of a $\sG$-stratum or in terms of a trivialization and a
$\hG$-stratum.

\begin{defn}
  Let $x \in \B(\hG)$ and let $r \ge 0$ be a real number.  A
  \emph{$\hG$-stratum} of depth $r$ is a triple $(x, r, \b)$ such that
  $\b \in (\hfg_{x, r}/\hfg_{x, r+})^\vee$.
\end{defn}

One can similarly define $\sG$-strata associated to points in
$\B(\Aut(\sG))$.  For a stratum $(x,r,\b)$ of this type,
$x\in\B(\Aut(\sG))$ and $\b \in ((\fg_\sG)_{x, r}/(\fg_\sG)_{x,
  r+})^\vee$.

Recall from geometric invariant theory that if $W$ is a representation
of a reductive group $H$, then a point $w\in W$ is called unstable if
$0$ is in the Zariski closure of the orbit $H\cdot w$; otherwise, it
is semistable.  In characteristic zero, $w$ is unstable if and only if
there exists a one-parameter subgroup $\g:\Gm\to H$ such that
$\lim_{t\to 0}\g(t)\cdot w=0$~\cite{Kempf}.  For example, $X\in\hfg$
is unstable if and only if it is nilpotent.  We call a
functional in $\hfg^\vee$ to be nilpotent if it is unstable.

\begin{defn} \label{fund} We say that a stratum $(x,r,\b)$ is
  \emph{fundamental} if the functional $\b$ is a semistable point of
  the $\hG_x/\hG_{x+}$-representation $(\hfg_{x, r}/\hfg_{x,
    r+})^\vee$.
\end{defn}
 Note that a $\GL(\hV)$-stratum is the same thing as a
  $\hV$-stratum in the sense of Definition~\ref{ustrat}.

\begin{rmk} Observe that  $(x,r,\b)$ can only be fundamental when $r\in\Crit_x$.
\end{rmk}

We will give some equivalent conditions which are easier to compute in
Proposition~\ref{nilpotent}.  For example, we will see in
Proposition~\ref{dual} that $(\hfg_{x, r}/\hfg_{x,
  r+})^\vee$ may be identified with $\hfg^\vee_{x, -r} / \hfg^\vee_{x, -r+}$.  We will
call $\tb \in \hfg^\vee_{x, -r}$ a \emph{representative} for $\b$ if
$\b$ corresponds to $\tb+\hfg^\vee_{x, -r+}\in\hfg^\vee_{x, -r} /
\hfg^\vee_{x, -r+}$.  Then, $(x,r,\b)$ is fundamental if and only if
every representative $\tb$ is nonnilpotent.

We now show how to associate strata to formal flat $G$-bundles.
Recall that $\hfg^\vee\cong \Om^1(\hfg)$ and similarly
$\fg_\sG^\vee\cong \Om^1(\fg_\sG)$.  In both cases, we will let
$X_{\tb}$ be the one-form corresponding to the functional $\tb$.

\begin{defn}\label{stratadefabstract}
  A flat $G$-bundle $(\sG, \n)$ contains the $\sG$-stratum $(x, r, \b)$
  if, for any (or
  equivalently, some) representative $\tb \in {\fg_\sG^\vee}_{x, -r}$ of
  $\b$,
\begin{equation}\label{stratadefabstracteq}(\n- i \ddz - X_{\tb})((V_\sG)_{x, i}) \subset \Om^1(V_{\sG})_{x,
  (i-r)+}, \quad\forall i\in\R, V\in\Rep(G).
\end{equation} 
\end{defn}

It will be convenient to make use of a similar concept involving an
explicit trivialization.
\begin{defn}\label{stratadef}
  A flat $G$-bundle $(\sG, \n)$ contains the $\hG$-stratum $(x, r, \b)$
  with respect to the trivialization $\phi$ if, for any (or
  equivalently, some) representative $\tb \in \hfg^\vee_{x, -r}$ of
  $\b$,
\begin{equation}\label{stratadefeq}(\n_\phi - i \ddz - X_{\tb})(\hV_{x, i}) \subset \Om^1(\hV)_{x,
  (i-r)+}, \quad\forall i\in\R, V\in\Rep(G).
\end{equation} 
\end{defn}
Lemma~\ref{strataequivariance} states that  if $(\sG, \n)$ contains the stratum $(x, r, \b)$ with
respect to $\phi$, then it contains $(gx,r,g\b)$ with respect to
$g\phi$. 

Since $X_{\tb}(\hV_{x, i}) \subset \Om^1(\hV)_{x, i-r}$ and
$i\ddz(\hV_{x, i}) \subset \Om^1(\hV_{x,i})\subset\Om^1(\hV)_{x,
  i-r}$, an immediate consequence of \eqref{stratadefeq} is that
$\n_\phi (\hV_{x, i}) \subset \Om^1(\hV)_{x, i-r}$ for all $i$.  Also,
note that the term $i\ddz$ can be omitted from \eqref{stratadefeq} if
$r>0$.  Similar considerations apply to Definition~\ref{stratadefabstract}.

As we will see in Proposition~\ref{strataprop}, when $x\in\Ao$,
stratum containment with respect to $\phi$
is equivalent to a concrete and easily verified condition on the
matrix $[\n]_\phi$.  

Just as for flat vector bundles, there is a notion of regular and
irregular singularity for a flat $G$-bundle.  
\begin{defn}
  We say that a formal flat $G$-bundle $(\sG, \n)$ is \emph{regular
    singular} if for all representations $V$, the associated flat
  vector bundle $(V_\sG, \n_V)$ is regular singular.  Otherwise, it is
  called \emph{irregular singular}.
\end{defn}

As one would hope, a flat $\GL_n$-bundle is regular singular if and
only if the corresponding flat vector bundle is regular singular.
This is indeed true; see Corollary~\ref{glslope}.

It is more subtle to define the \emph{slope} of $(\sG,\n)$.  One of
the major results of this paper is using the theory of fundamental strata to
provide the appropriate generalization.

\subsection{Main theorems}\label{main}

We can now state the main results of the paper.  The proofs are given
in Section~\ref{proofs}.

In the following theorem, $\Ao\subset\B(\hG)$ is the apartment
corresponding to the split maximal torus $\hT\subset\hG$, where $T$ is
a fixed maximal torus in $G$.  We call a point in $\B(\hG)$
(resp. $\B(\Aut(\sG))$) rational if for every $V\in\Rep(G)$, the
critical numbers for the filtration on $\hV$ (resp. $V_\sG$) are
rational.  We denote the set of rational points in $\B(\Aut(\sG))$,
$\B(\hG)$, and $\Ao$ by $\B(\Aut(\sG))^{\mathrm{rat}}$,
$\B(\hG)^{\mathrm{rat}}$, and $\Ao^{\mathrm{rat}}$ respectively.  We
call a stratum rational if it is based at a rational point in the
building.  \emph{Optimal points} are certain points in $\Ao$ that
generalize the barycenters of simplices in the reduced building for
$\GL_n$.  (See Section~\ref{proofs} for the definition.)  One of their
nice properties is that they lie in $\Ao^{\mathrm{rat}}$.

We only wish to consider nonrational strata when $\R\subset k$.
Accordingly, throughout the paper, we adopt the following convention:
\begin{conv}  All strata are assumed to be rational without further
  comment unless $\R\subset k$.
\end{conv}

\begin{thm}\label{MP}
  Every flat $G$-bundle $(\sG, \n)$ contains a fundamental
  $\sG$-stratum $(x,r,\b)$ with $x\in\B(\Aut(\sG))^{\mathrm{rat}}$;
  the depth $r$ is positive if and only if $(\sG, \n)$ is irregular
  singular.  More precisely, $x$ can be chosen to be the preimage of
  an optimal point in $\Ao$ under some trivialization
  $\phi$. Moreover, the following statements hold.
\begin{enumerate}
\item\label{MP1} If $(\sG, \n)$ contains the stratum $(y,r', \b')$, then
$r' \ge r$.  
\item \label{MP3} If $(\sG,\n)$ is irregular singular, a stratum
  $(y,r', \b')$ contained in $(\sG, \n)$ is fundamental if and only if
  $r' = r$.
\end{enumerate}
\end{thm}

 It is an important question to understand the set of strata contained
 in a given flat $G$-bundle.  As a first step in this direction, we
 have shown that two such strata of the same depth $r$ are
 \emph{associates} of each other.  The formal definition is given in
 Definition~\ref{assdef}.

\begin{thm}\label{associateness}  Suppose
  that $(\sG, \n)$ contains the $\hG$-strata $(x, r, \b)$ and
  $(y, r, \b')$ with respect to the trivializations $\phi$ and $\phi'$
  respectively.  Then, $(x, r, \b)$ and $(y, r, \b')$ are associates
  of each other.  In particular, all fundamental strata contained in
  $(\sG, \n)$ are associates of each other.
\end{thm}

We now define the slope of a flat $G$-bundle.

\begin{defn} The \emph{slope} of the flat $G$-bundle $(\sG,\n)$ is the depth
  of any fundamental stratum contained in $(\sG,\n)$.
\end{defn}

This definition makes sense by Theorem~\ref{MP}.  It also follows
that the slope is always an \emph{optimal number} in the sense of
\cite{AD}, i.e., a critical number for the filtration on $\hfg$ determined
by an optimal point.

\begin{thm}\label{slopethm} The slope of the flat $G$-bundle $(\sG,\n)$ is a
  nonnegative rational number.  It is positive if and only if
  $(\sG,\n)$ is irregular singular.  The slope may also be
  characterized as
\begin{enumerate}
\item  the minimum depth of any stratum contained in
  $(\sG,\n)$; 
\item  the minimum depth of any stratum contained in
  $(\sG,\n)$ and based at an optimal point; 
\item the maximum slope of the associated flat vector bundles
  $(V_\sG,\n_V)$; or
\item the maximum slope of the flat vector bundles associated to the
  adjoint representations and the characters.
\end{enumerate}
\end{thm}

\begin{rmk} Since the category of formal flat $\GL_n(V)$-bundles is
  equivalent to the category of rank $n$ flat vector bundle, one would
  expect the notions of strata, stratum containment, and slope on the
  two categories to correspond.  This is indeed the case; see
  Corollaries~\ref{vbgl} and \ref{glslope}.
\end{rmk}

  If $E$ is a degree $e$ field extension of $F$, then $E$ is generated
  by an element $u_E$ with $u_E^e=z$.  We let $\pi_E : \Spec (E) \to
  \Spec(F)$ be the associated map of spectra.  The pullback of a flat
  $G$-bundle $(\sG,\n)$ to $\Spec(E)$ will be denoted by $(\pi_E^*\sG,
  \pi_E^*\n)$.

\begin{lem}\label{slopepullback} If a flat $G$-bundle $(\sG,\n)$ over
  $\Spec(F)$ has slope $r$, then $(\pi_E^*\sG, \pi_E^*\n)$ has slope
  $[E:F]r$.
\end{lem}

\begin{prop}\label{FGslope}
  A flat $G$-bundle $(\sG, \n)$ has slope $r$ if and only if there
  exists a finite field extension $E/F$ and a trivialization $\phi$ of
  the pullback of $(\sG,\n)$ to $\Spec(E)$ such that
  \begin{equation} \label{stdform} (\pi_E^*[\n])_\phi = \sum_{i =
      -n}^\infty M_i u^{i} \ddu,
\end{equation}
with $M_i \in \fg$, $M_{-n}$ is nonnilpotent, and $r=n/[E:F]$.
\end{prop}

 \begin{rmk}\label{sloperemark}

   In \cite[Section 5]{FrGr}, Frenkel and Gross suggest the condition
   of Proposition~\ref{FGslope} as a definition of the slope of a flat
   $G$-bundle.  The corollary shows that this approach is independent
   of the choices and hence provides an alternate description of the
   slope. (This can also be shown directly using results from Section
   $9$ of \cite{BaVa1}.)

   We remark that there are advantages to the approach to the slope
   taken in this paper.  For example, there is no need to pass to a
   ramified cover in order to put $[\n]$ into the appropriate form;
   indeed, there is an explicit algorithm appearing in the proof of
   Theorem~\ref{MP} that allows one to find a trivialization of $\sG$
   with respect to which $(\sG, \n)$ contains a fundamental stratum.
   Moreover, one is able to predict which rational numbers occur as
   the slope of a flat $G$-bundle, since these must be optimal numbers
   for $\hG$.  Finally, a fundamental stratum provides additional
   structural information about a formal flat $G$-bundle beyond the
   slope.  This is explored further in~\cite{BrSa5}.

   There is also a recent preprint of Chen and Kamgarpour which shows
   how the slope may be characterized using \emph{opers}~\cite{ChKa}.
   More specifically, they define the slope of an oper and show that
   it agrees with the slope of the underlying flat $G$-bundle, using
   the formulation of \cite{FrGr}.  Since any formal flat $G$-bundle
   has an oper structure~\cite{FrZh}, its slope can be given as the
   slope of any associated oper.  However, unlike fundamental strata,
   there is no known algorithm for effectively computing the oper
   structure on a flat
   $G$-bundle.

\end{rmk}

\section{Moy-Prasad filtrations and strata}\label{MPfilt}

We continue with the notation of Section~\ref{Prelim}, so $G$ is a
connected reductive group over $k$ with Lie algebra $\fg$.  All
results about Bruhat-Tits buildings, Moy-Prasad filtrations, and
strata will be stated only for $\hG$, but analogous statements hold
throughout for $\Aut(\sG)$.  To simplify notation, we will write
$\B=\B(\hG)$ and $\bB=\bB(\hG)$.

\subsection{Filtrations on loop group representations}\label{subfilt}

In this section, we will discuss some further properties of Moy-Prasad
filtrations.

Recall that $\Ao\cong X_*(T)\otimes\R$ is the standard apartment of
$\B$ determined by the split maximal torus $\hT$.
As long as $\R\subset k$, points in $\Ao$ may also be viewed as
elements of $\ft$.  The Lie algebra $\ft'$ of any split
torus $T'$ over $k$ is canonically isomorphic to $X_*(T')\otimes_\Z
k$; the map is induced by sending a cocharacter $\lambda$ to
$d\lambda(1)$.  If $k'$ is a subfield of $k$, we let $\ft'_{k'}$ be
the image of $X_*(T')\otimes_\Z k'$ under this map.  Assuming that
$\R\subset k$, we have $\Ao\cong\ft_\R$, and we write $\tx\in\ft_\R$
for the image of $x\in\Ao$.

\begin{rmk} A point $x\in\B$ is called rational if
  $\Crit_x(V)\subset\Q$ for all $V\in\Rep(G)$.  Every element of
  $\B^{\mathrm{rat}}$ is conjugate to a point in $\Ao^{\mathrm{rat}}$,
  and $\Ao^{\mathrm{rat}}$ consists precisely of those elements of
  $\Ao$ coming from $X_*(T)\otimes_\Z \Q$.  In particular, the map
  $x\mapsto\tx$ defined above when $\R\subset k$ always makes sense
  when restricted to $\Ao^{\mathrm{rat}}$; it gives an isomorphism
  $\Ao^{\mathrm{rat}}\cong\ft_\Q$.  We adopt the convention that when
  the notation $\tx$ for $x\in\Ao$ is used, it will automatically mean
  that $x\in\Ao^{\mathrm{rat}}$ unless $\R\subset k$.
\end{rmk}

There is a very useful alternate description of Moy-Prasad filtrations for $x\in\Ao$ in terms of operators
related to the Euler vector field $\tau= z \frac{\pd}{\pd z}$. 

\begin{prop}\label{ev}   Let $V$ be 
  a finite-dimensional representation of $G$, and fix $x \in \Ao$ and
  $r \in \R$.
\begin{enumerate}
\item \label{ev1} The space $\hV_x(r)$ is the eigenspace 
 corresponding to
the eigenvalue $r$ in $\hV$ for the differential operator $\tau + \tx$.
\item\label{ev2} An element $v \in \hV$ lies in $\hV_{x, r}$ if and
  only if $(\tau + \tx) (v) - r v \in \hV_{x, r+}$.
\item\label{ev3} The set $\hV_x(r)$ constitutes a full set of coset representatives for 
the coset space $\hV_{x, r} / \hV_{x, r+}$.
\item If $X\in\hfg_x(s)$, then $\ad(X)(\hV_x(r))\subset \hV_x(r+s)$.
\end{enumerate}
\end{prop}
\begin{proof}
Each of these statements follows immediately from the definitions of $\hV_x (r)$
and $\hV_{x, r}$.  
\end{proof}

The following lemma allows one to express the action of $\hN=N_F$ on
$\Ao$ in term of the differential calculus above.

\begin{lem}\label{bnorm}
Suppose that $n \in \hN$ is a coset representative of $w \in \hW$.
 For all $x \in \Ao$, $\Ad(n) (\tx) - \tau (n) n^{-1} \in \widetilde{w x}
+ \hft_{0+}.$
 \end{lem}
\begin{proof}
  Write $n = t n'$, with $t \in \hT$ and $n' \in N$.  Since $\hT/
  \hT_0 \cong X_*(T)$ (as abstract groups), we may write $t = z^y t'$
  for some $y \in X_*(T)$ and $t'\in\hT_0$.  An application of the
  Leibniz rule shows that
\begin{equation}\label{templabel}
\tau(n) n^{-1} = \Ad(t) \left[\tau (n')(n')^{-1}\right] + \tau (t) t^{-1} =
\tau(z^y) z^{-y} + \tau(t') (t')^{-1} \in \tilde{y} + \hft_{0+}.
\end{equation}
Here, $\tau(n') (n')^{-1} = 0$, since $n' \in G$.

Define $a \in \Ao$ by $\ta = \Ad(n) (\tx) - \ty.$ Suppose that $v \in
\hV_{x,r}$ for some representation $V$, and let $u = n v \in \hV_{w
  x,r}$.  We will show that
\begin{equation}\label{lemexp}
\tau (u) + \ta u \in r u + \hV_{w x, r+}.
\end{equation}
Assuming this, Proposition~\ref{ev}\eqref{ev2} implies that $(\ta-\widetilde{w
  x}) (u) \in \hV_{w x, r+}$.  Furthermore, if this expression holds
for arbitrary $V$, $r$, and $u$, then $\ta =\widetilde{w x}$.
(Otherwise, let $V$ be a faithful representation.  Then, there exists
$r \in \R$ and an eigenvector in $\hV_{wx, r}/\hV_{wx, r+}$ with
nonzero eigenvalue for the action of $\ta-\widetilde{w x}$, a
contradiction.)

It remains to prove \eqref{lemexp}. Calculating, we obtain
\begin{align*}
  \tau (n v) + \ta u & = \tau (n) n^{-1} u + n (\tau (v)) + \ta u
  &&\text{by the Leibniz rule}
  \\
  & \in  n(\tau(v)) +(\Ad(n)\tx) (u) + \hft_{0+} u  &&\text{by \eqref{templabel}}\\
  & = n ((\tau +\tx) (v)) + \hft_{0+} u &&\text{by definition of $a$}\\
  & \subset r n (v) + n \hV_{x, r+} = r u + \hV_{w x, r+}&&\text{by Proposition~\ref{ev}\eqref{ev2}}.
\end{align*}

\end{proof}

Next, we discuss further the quotient group $\hG_x/\hG_{x+}$.  This
group is isomorphic to a reductive maximal rank subgroup of $G$.  We
will construct an explicit isomorphism.  We will first define this
isomorphism for $x\in\Ao$ and then will apply equivariance to define
the map in the general case.

Fix $x\in\Ao$.  Let $\fh_x\subset \fg$ be the subalgebra
$\fh_x=\ft\oplus\bigoplus_{\{\a\mid d\a(\tx)\in\Z\}}\fu_\a$.
It is the Lie algebra of the reductive subgroup $H_x\subset G$
generated by $T$ and the corresponding root subgroups $U_\a$.  The
group $H_x$ is the connected centralizer of $\exp (2 \pi i \tx)\in
G$.\footnote{If $k$ contains $\R$ (and hence $\C$), then $\exp (2 \pi
  i \tx)$ makes sense; otherwise, by our convention, $\tx\in \ft_\Q$.}
Note that $H_x$ only depends on the image of $x$ in $\bB$.  

Next, we define a $T$-equivariant homomorphism
$\theta'_x:H_x\to\hG_x$.  Choose isomorphisms $u_\a:\Ga\to U_\a$
giving a realization of the root system of $H_x$ with respect to $T$
in the sense of Springer~\cite[Section 8.1]{Spr}.  On the generating
subgroups, $\theta'_x$ is defined via $T\hookrightarrow T(\fo)$ and
$\theta'_x(u_\a(c))=u_\a(c z^{-\a(\tx)})$ for $c\in k$.  It is easy to
check that $\theta'_x$ satisfies the relations among the generators
(see for example~\cite[Section 9.4]{Spr}) and hence is a homomorphism.
(If $\tx\in\ft_\Q$, then $\theta'_x$ is just conjugation by
$z^{-\tx}$; here, if $\l_1,\dots\l_n$ is a basis for $X_*(T)$ and
$\tx=\sum d\l_i(\tx_i)$, then $z^{-\tx}=\prod \l_i(z^{-\tx_i})\in
T(\bF)$.)

Let $\theta_x:H_x\to\hG_x/\hG_{x+}$ be the induced map.  Since
\begin{equation*}
\hfg_x/\hfg_{x+}\cong\hfg_x(0)=\ft \oplus \bigoplus_{\a:d\a(\tx)\in\Z}\fu_\a
z^{-\a(\tx)},
\end{equation*}
it is clear that $d\theta_x$ is an isomorphism.  The map $\theta_x$ is
thus an isogeny which restricts to an isomorphism on the maximal torus
$T$, hence is an isomorphism.  The group $H_x$ acts on $\hV_x(r)$ for
any $r$.  It is now easy to see that $\theta_x$ intertwines the
representations $\hV_x(r)$ and $\hV_{x,r}/\hV_{x,r+}$.

If $y=gx$ for some $g\in\hG$, then $\theta_x$ composed with
conjugation by $g$ gives an isomorphism
$\theta_{x,g}:H_x\overset{\sim}{\to}\hG_y/\hG_{y+}$.  This map depends
on the choice of $g$.  If $y=g'x$ also, then $g'=gh$ for some
$h\in\Stab(x)$, and $\theta_{x,g'}$ and $\theta_{x,g}$ differ by an
automorphism of $H$.  However, if $\Stab(x)=\hG_x$, then
$\theta_{x,g'}$ is just conjugation by $\theta_x^{-1}(h\hG_{x+})$
composed with $\theta_{x,g}$.  Thus, in this case, the isomorphism is
well-defined up to an inner automorphism of $H_x$.  In particular,
this is true when $x$ lies in a minimal facet (i.e., $\hG_x$ is a
maximal parahoric) or $G$ is semisimple and simply
connected~\cite[p.50]{Ti}.  We sum up this discussion in the following
proposition.

\begin{prop}\label{r0stratum} Fix $x\in\Ao$. 
\begin{enumerate}\item\label{r0stratum1} The map $\theta_x:H_x\to\hG_x/\hG_{x+}$
    is an isomorphism that intertwines the representations
    $\hV_{x}(r)$ and $\hV_{x,r}/\hV_{x,r+}$ for all $r$.
  \item If $\Stab(x)=\hG_x$, the maps $\theta_{x,g}$ induce a
    compatible family of bijections between the sets of conjugacy
    classes of $\hG_y/\hG_{y+}$ for $y\in \hG\cdot x$ and the set of
    conjugacy classes of $H_x$.
  \end{enumerate}
\end{prop}

Moy-Prasad filtrations are well-behaved with respect to the usual
algebraic operations.  The filtration on the trivial representation
$\hat{k}=F$ is just the usual degree filtration; it is independent of
$x\in\B$.  If $W$ is a subrepresentation of $V$, then
$\hatW_{x,r}=\hV_{x,r}\cap \hatW$.  If $V$ and $W$ are two
representations, then $\widehat{(V\oplus W)}_{x,r}=\hV_{x,r}\oplus
\hat{W}_{x,r}$.  Next, under the identification
$\widehat{\Hom_k(V,W)}\cong\Hom_F(\hV,\hat{W})$, we have
\begin{equation}\label{hommp} \widehat{\Hom_k(V,W)}_{x,r}=\{f\mid
  f(\hV_{x,s})\subset \hat{W}_{x,s+r} \text{ for all } s\in\R\}.
\end{equation}
If we let $V^\vee=\Hom(V,k)$ and denote the $F$-linear dual of an
$F$-vector space $U$ by $U^*$, then this formula also gives the
filtrations on $\widehat{(V^\vee)}=\hV^*$ and on $\hV\otimes_F
\hatW\cong \widehat{V\otimes W}\cong\widehat{\Hom_k(V^\vee,W)}$.

We have already seen how to define Moy-Prasad filtrations on the space of
one-forms $\Om^1(\hV)$; we set $\Om^1(\hV)_{x,r}$ equal to the image
of $\hV_{x,r}$ under the isomorphism $\hV\to\Om^1(\hV)$ given by
$u\mapsto u\ddz$.  While this map depends on the choice of
uniformizer, it is easy to see that the filtrations do not.

We have also defined Moy-Prasad filtrations on the smooth $k$-dual of
$\hfg$.  To do this, we identified $\hfg$ and $\hfg^\vee$ using the
fixed $G$-invariant form $\left<, \right>$ and defined filtrations on
the dual space by transport of structure.  For completeness, we will
show how to define filtrations on smooth dual spaces $\hV^\vee$, where
$V\in\Rep(G)$ is not necessarily endowed with an appropriate invariant
form.  The new approach will of course give the same filtrations on
$\hfg^\vee$.

If $V$ is a representation of $G$, we define an isomorphism
$\hV^*=\widehat{(V^\vee)}\overset{\kappa}{\to} (\hV)^\vee$ via
$\kappa(\a)(v)=\Res \a(v)\ddz$, and set $(\hV^\vee)_{x,
  r}=\kappa({\hV^*}_{x,r})$. Again, although the map depends on the
uniformizer, the filtrations do not.  (To see this, simply observe
that $\kappa$ is the composition of the map $\hV^*\to
\Om^1(\widehat{V^\vee})$ with the canonical isomorphism
$\Om^1(\widehat{V^\vee})\to(\hV)^\vee$ sending $\omega$ to
$v\mapsto\Res \omega(v)$.)

If one further supposes that $V$ is endowed with a nondegenerate
$G$-invariant symmetric bilinear form $\left(,\right)$, then the
resulting isomorphism $V\to V^\vee$ induces an isomorphism $\hV\to
(\hV)^\vee$ given by $v\mapsto \Res\left(v,\cdot\right)\ddz$.  We
denote the corresponding $k$-bilinear pairing $\hV\times\hV\to k$ by
$\left(,\right)_\ddz$.

\begin{rmk}\label{coadjoint} In the case of the adjoint representation, the
 isomorphism $\Om^1(\hfg) \cong\hfg^*$ intertwines the $G$-action
 defined in Section   with the coadjoint action.
\end{rmk}

We now give the basic facts about the relationship between duality and
the Moy-Prasad filtration.  For the adjoint representation, these
results appear in \cite[Sections 3.5 and 3.7]{MP1}.

If $W$ is a $k$-subspace of $\hV$, we let $W^\perp=\{\phi\in
\hV^\vee\mid \phi(W)=0\}$.

\begin{prop}\label{dual} Let $V$ be a finite-dimensional representation of $G$.
  Fix $x\in\B$ and $r\in\R$.
\begin{enumerate}\item\label{dual1} 
The Moy-Prasad filtrations on
  $\hV^\vee$ may be expressed in terms of the annihilators of the
  filtrations on $\hV$ as $\hV^\vee_{x,-r}= \hV^\perp_{x,r+}$ and
  $\hV^\vee_{x,-r+}= \hV^\perp_{x,r}$.
\item\label{dual2} There is a natural $\hG_x$-invariant perfect pairing
  \begin{equation*}
 \hV^\vee_{x, -r}/\hV^\vee_{x, -r+} \times \hV_{x, r} / \hV_{x, r+} \to k,
 \end{equation*} 
 which induces the isomorphism $(\hV_{x, r}/ \hV_{x, r+})^\vee \cong
 \hV^\vee_{x, -r}/\hV^\vee_{x, -r+}$.
\item \label{dual3} There are $\hG_x$-isomorphisms $(\hV_{x,r})^\vee\cong
  \hV^\vee/\hV^\vee_{x,-r+}$ and $(\hV_{x,r+})^\vee\cong
  \hV^\vee/\hV^\vee_{x,-r}$.
\item \label{dual4} Suppose that $V$ is endowed with a nondegenerate
  $G$-invariant symmetric bilinear form $\left(,\right)$.  Then,
  $\left(,\right)_\ddz$ induces $\hG_x$-isomorphisms
  $\hV^\vee_{x,-r}\cong\hV_{x,-r}\cong\Om^1(\hV)_{x,-r}$ and
  $\hV^\vee_{x,-r+}\cong\hV_{x,-r+}\cong\Om^1(\hV)_{x,-r+}$; in particular, $(\hV_{x, r}/
  \hV_{x, r+})^\vee \cong \hV_{x, -r}/\hV_{x, -r+}$.
\end{enumerate}
\end{prop}

\begin{proof} If $x\in\Ao$, these statements are easily checked using
  the gradings.  For example, to prove \eqref{dual1}, one need only
  observe that the natural pairing
  $\hV^\vee_{x}(s)\times \hV_x(r)\to k$ is perfect if $s=-r$ and $0$
  otherwise.  (Here, the graded components of $\hV^\vee_{x}$ are
  defined using transport of structure from $\hV^*$ via $\kappa$.)
  Parts \eqref{dual2} and \eqref{dual3} then follow from
  \eqref{dual1}, and part \eqref{dual4} is a consequence of
  \eqref{dual2} and \eqref{dual3}.

  The general case is obtained by conjugating $x$ into $\Ao$.
\end{proof}

In the situation of part~\eqref{dual4}, we obtain
$\Crit_x(V^\vee)=-\Crit_x(V)$.  If in addition $\fz$ acts trivially on
$V$, $\Crit_x(V^\vee)=\Crit_x(V)$; in particular, $\Crit_x(\fg^\vee)=\Crit_x$.

\subsection{Fundamental Strata}\label{fundsect}

In this section, we provide more information about strata.  Recall
that a $\hG$-stratum of depth $r$ is a triple $(x, r, \b)$ such that
  $\b \in (\hfg_{x, r}/\hfg_{x, r+})^\vee$.  We denote the set of
  $G$-strata (resp, $G$-strata of depth $r$) by $\sS^G$ (resp. $\sS^G_r$).

  By Proposition~\ref{dual}, we may identify
  $(\hfg_{x, r}/\hfg_{x, r+})^\vee$ with
  $\hfg^\vee_{x, -r} / \hfg^\vee_{x, -r+}$.  We will call
  $\tb \in \hfg^\vee_{x, -r}$ a \emph{representative} for $\b$ if $\b$
  corresponds to
  $\tb+\hfg^\vee_{x, -r+}\in\hfg^\vee_{x, -r} / \hfg^\vee_{x, -r+}$.
  If $x\in\Ao$, we let $\tbo$ denote the unique homogeneous
  representative in $\hfg^\vee_x(-r)$.

The $\hG$-equivariance of Moy-Prasad filtrations induces a natural
action of $\hG$ on $\sS^G$ and on each $\sS^G_r$.  Indeed, the
coadjoint action induces a map $\bAd^*(g):\hfg^\vee_{x, -r} /
\hfg^\vee_{x,-r+}\to \hfg^\vee_{g x, -r} / \hfg^\vee_{g x,-r+}$.  If
we let $g\b\in(\hfg_{gx,r}/\hfg_{gx,r+})^\vee$ be the functional
induced by $\bAd^*(g)(\tb)$, then the action on strata is defined by
$g\cdot(x,r,\b)=(gx,r,g\b)$.

Definition~\ref{fund} states that a stratum $(x,r,\b)$ is fundamental
if the functional $\b$ is a semistable point of the
$\hG_x/\hG_{x+}$-representation $(\hfg_{x, r}/\hfg_{x, r+})^\vee$.
This condition may be expressed more explicitly as follows.

\begin{prop}\label{nilpotent} The stratum $(x, r, \b)$ is
  nonfundamental if and only if the coset $\tb + \hfg^\vee_{x, -r+}$
  contains a nilpotent element.  Moreover, if $x\in\Ao$, $(x, r, \b)$
  is nonfundamental if and only if the graded representative $\tbo$ is
  nilpotent.
 \end{prop}

\begin{proof}
  Since the stratum $(x, r, \b)$ is fundamental if and only if $g(x,
  r, \b)$ is fundamental, it suffices to assume that $x\in\Ao$.
  Suppose that $(x, r, \b)$ is nonfundamental.  By
  Proposition~\ref{r0stratum}\eqref{r0stratum1}, the homogeneous
  representative $\tbo$ is an unstable point of the
  $H_x$-representation $\hfg^\vee_x(-r)$.  Let $\g:k^*\to H_x$ be a
  one-parameter subgroup for which $\lim_{t \to 0} \g(t)\cdot\tbo =
  0$.  Let $\tg:F^*\to\hG$ be the one-parameter subgroup obtained from
  $\g$ by extension of scalars.  Since $\tbo$ lies in the sum of
  positive weight spaces for $\g$, it also is in the sum of positive
  weight spaces for $\tg$.  Thus, $\lim_{t \to 0} \tg(t)\cdot\tbo =
  0$, and $\tbo$ is nilpotent.

  Next, if $\tb + \hfg^\vee_{x, -r+}$ contains a nilpotent element,
  then $\b$ is unstable by~\cite[Proposition 4.3]{MP1}.  Finally, if
  $\tbo$ is nilpotent, then $\tb + \hfg^\vee_{x, -r+}$ of course
  contains a nilpotent element.

\end{proof}
\begin{rmk} \label{fundrmk} The proof of the proposition gives a
  practical way of determining whether a stratum is fundamental.  One
  conjugates into the standard apartment (or any apartment
  corresponding to a maximal torus of $G$) and checks to see if the
  graded component of the functional is nonnilpotent.
\end{rmk}

We conclude this section by defining what it means for two strata to
be \emph{associates} of each other.  We have seen in
Theorem~\ref{associateness} that two fundamental strata contained in a
flat $G$-bundle are associates of each other. Before giving the
definition, we need a proposition.

\begin{prop}\label{deltaprop}
  Suppose that $x, y \in \B$.  There exists a element $\de_{x, y} \in
  \hfg^\vee_{x,0} \cap \hfg^\vee_{y,0}$ with the following property:
  for all $ g \in \hG$ such that $g x, g y \in \Ao$,
  \begin{equation}\label{delta}
\Ad^*(g)(\de_{x, y}) \subset (\widetilde{gy} - \widetilde{gx}) \ddz + 
\hfg^\vee_{gx, 0+} + \hfg^\vee_{gy, 0+}.
\end{equation}
\end{prop}

The proof will be deferred to Section~\ref{proofs}.  When $x,y\in\Ao$,
it will follow from the proof that we may take $\de_{x, y} =
(\tx-\ty)\ddz$.  We will usually do so without comment.

\begin{defn}\label{assdef}
Let $(x, r, \b)$ and $(y,r, \b')$ be two $G$-strata.  Choose
representatives $\tb \in \hfg^\vee_{x, -r}$ and
$\tb' \in  \hfg^\vee_{y, -r}$ for $\b$ and $\b'$, respectively.
We say that $(x, r, \b)$ and $(y, r, \b')$ are \emph{associates} of
each other
if there exists $g \in \hG$ such that
\begin{equation}\label{asseq}
\left(\Ad^*(g)(\tb) + \hfg^\vee_{g x, -r+}\right) \cap 
\left(\tb' - \de_{g x, y} + \hfg^\vee_{y, -r+}\right) \ne \emptyset.
\end{equation}
Note that when $r > 0$, the $\de_{g x, y}$ term is unnecessary.
\end{defn}

It is immediate that this definition is independent of the particular
elements $\de_{x, y}$ chosen to satisfy~\eqref{delta}.  It is also
symmetric.  We observe further that conjugate strata are associates.
Indeed, if $(y,r,\b')=(gx,r,g\b)$, then the two cosets coincide.  (One
may take $\de_{y,y}=0$.)

\section{Formal flat $G$-bundles and fundamental strata}\label{formfund}

Recall from Section~\ref{flatg} that a formal flat $G$-bundle is a
principal $G$-bundle $\sG$ over $\Dx$ endowed with a connection $\n$.
Upon fixing a trivialization $\phi:\sG\to\hG$, the corresponding
connection on the trivial flat $G$-bundle can be written as
$\n_\phi=d+\nbr_\phi$ with $\nbr_\phi\in\Om^1(\hfg)$. Equivalently,
$(\sG,\n)$ can be viewed as a compatible family of flat vector bundles
$(V_\sG,\n_V)$ indexed by $\Rep(G)$.  In terms of the trivialization
$\phi$, the flat vector bundle associated to $(V,\rho)$ is
$(\hV,d+\rho(\nbr_\phi)$.  We remark that if $V$ is the standard
representation of $\GL_n$, then the functor
$(\sG,\n)\mapsto(V_\sG,\n_V)$ is an equivalence of categories between
flat $\GL_n$-bundles and flat rank $n$ vector bundles.

As discussed before Proposition~\ref{dual},
$\Om^1(\hfg)\cong\hfg^\vee$, where the isomorphism depends only on our
fixed choice of invariant form $\left<,\right>$.  Recall that $X_{\tb}\in\Om^1(\hfg)$
denotes the one-form corresponding to $\tb\in\hfg^\vee$.

\subsection{Strata contained in flat $G$-bundles}\label{strataing}

In Section~\ref{s:gstrata}, we showed how to associate $G$-strata to
flat $G$-bundles.  In particular, a flat $G$-bundle $(\sG, \n)$
contains the $\hG$-stratum $(x, r, \b)$ with respect to the
trivialization $\phi$ if, for any representative
$\tb \in \hfg^\vee_{x, -r}$ of $\b$,
$(\n_\phi - i \ddz - X_{\tb})(\hV_{x, i}) \subset \Om^1(\hV)_{x,
  (i-r)+}$
for all $i\in\R$ and $V\in\Rep(G)$.  More succinctly, this holds if for
all $(V,\rho)$, the induced flat vector bundle $(V_{\sG},\n_V)$
contains the $\hV$-stratum $(\rho(x),r,\rho(\b))$. 

In this section, we will explain some properties of stratum
containment.  In particular, we will prove an equivariance
result for containment of $\hG$-strata which shows that the
trivialization-free version of containment in terms of
$\sG$-strata (Definition~\ref{stratadefabstract}) is
well-defined.  We will also show that when $x\in\Ao$, $\hG$-stratum
containment with respect to the trivialization $\phi$
is equivalent to a concrete and easily verified condition on the
matrix $[\n]_\phi$.  

\begin{rmk} As one would expect, the notions of containment of strata
  in flat $\GL_n$-bundles and in rank $n$ vector bundles are the same.
  This is shown in Corollary~\ref{vbgl}.  The proof uses the fact that
  $\Rep(\GL_n)$ is generated by the standard representation $V$ in the
  sense that  every representation of $G$ may be obtained from $V$
   via some combination of duals, direct sums,
  tensor products, and subrepresentations. More generally, suppose
  that $\Rep(G)$ is generated (in the same sense) by
  $\{V_i\}\subset\Rep(G)$.  We show in
  Proposition~\ref{containops} that a flat $G$-bundle contains a
  stratum if and only if \eqref{stratadef} is satisfied for each
  $V_i$.
\end{rmk}

Stratum containment is well-behaved with respect to change of trivialization:
\begin{lem}\label{strataequivariance} If $(\sG, \n)$ contains $(x, r, \b)$
  with respect to $\phi$, then it contains $(gx,r,g\b)$ with respect
  to $g\phi$.
\end{lem}


 \begin{proof}
   Fix a representation $(V,\rho)$.  Changing the trivialization
   $\phi$ by $g$ changes $\n_\phi - i \ddz - X_{\tb}$, viewed as an
   element of $\End_k(\hV)\ddz$, to $\rho(g)(\n_\phi - i \ddz -
   X_{\tb})\rho(g)^{-1}$.  (From now on, we will omit the $\rho$'s
   from the notation.) However, the gauge change formula implies that
   $\n_{g\phi}=g\n_\phi g^{-1}$ and equivariance of Moy-Prasad
   filtrations gives $X_{g\tb}=gX_{\tb}g^{-1}$.  Accordingly,

 \begin{equation}\label{triv}
 \begin{aligned}
   g(\n_\phi - i \ddz -
   X_{\tb})g^{-1}\hV_{gx,i}&=g(\n_\phi - i \ddz -
   X_{\tb})g^{-1}g \hV_{x, i}\\ & = g (\n_\phi - i
   \ddz -
   X_{\tb})\hV_{x, i} \\
   & \subset g(\Om^1(\hV)_{x,(i-r)+})=\Om^1(\hV)_{gx,(i-r)+}.
 \end{aligned}
 \end{equation}

 \end{proof}

We now give a more explicit description of what is means for
a stratum to be contained in a flat $G$-bundle when the stratum is
based at a point in the standard apartment.  This characterization
will be useful in calculations throughout the rest of the paper.

\begin{prop}\label{strataprop} Let $x \in \Ao$.  
Then, $(\sG,  \n)$ contains the stratum $(x,r, \b)$ with respect to
the trivialization $\phi$ if and only if 
   $[\n]_\phi - \tx \ddz \in \hfg_{x,r+}^\perp$ and the coset 
   $\left([\n]_\phi - \tx \ddz \right)
  + \hfg_{x, -r+}^\vee$ determines the functional $\b \in (\hfg_{x, r}
  / \hfg_{x,r+})^\vee$.
\end{prop}

\begin{proof}

Assume that $\nbr_\phi$ satisfies the given conditions, and
take $V\in\Rep(G)$.  By the hypothesis, we can take $X_{\tb}=[\n]_\phi - \tx \ddz
\in \Om^1(\hfg)_{x,-r}$, so $\n_\phi - i \ddz -
X_{\tb}=d+[\n]_\phi-i\ddz-([\n]_\phi - \tx \ddz)=d+\tx\ddz-i\ddz$.
Since Proposition~\ref{ev}\eqref{ev2} implies
\begin{equation}\label{speq}
  (d+\tx\ddz-i\ddz)(\hV_{x,i})\subset\Om^1(\hV)_{x,i+}\subset
  \Om^1(\hV)_{x,(i-r)+},
\end{equation}
the defining property \eqref{stratadefeq} is satisfied.

For the converse, suppose that $(\sG, \n)$ contains the stratum
$(x,r,\b)$ with respect to the trivialization $\phi$, and let $V$ be a
faithful representation.  As usual, take $\tb \in \hfg^\vee_{x, -r}$ a
representative for $\b$ with $X_{\tb} \in \Om^1 (\hfg)_{x, -r}$ the
corresponding element.  We then have \begin{equation*}
 (d + [\n]_\phi - i \ddz - X_{\tb}) \hV_{x, i} \subset \Om^1(\hV)_{x, (i -r)+}
\end{equation*}  Subtracting \eqref{speq}, we obtain
$([\n]_\phi - \tx \ddz - X_{\tb}) \hV_{x, i} \subset \Om^1(\hV)_{x,
  (i-r)+}$ for all $i \in \R$.  The faithfulness of $V$ now implies
that $[\n]_\phi - \tx \ddz - X_{\tb} \in \Om^1(\hfg)_{x, -r+}$.  In
particular, $X_{\tb} \in ([\n]_\phi - \tx \ddz) + \Om^1(\hfg)_{x,
  -r+}$.  Viewing $[\n]_\phi - \tx \ddz$ as a functional, we see that
it lies in it $\hfg_{x, -r}^\vee\cong\hfg_{x,r+}^\perp$ and
determines the same coset in $\hfg_{x, -r}^\vee/\hfg_{x, -r+}^\vee$
as $\tb$.  This proves the result.

\end{proof}

\begin{rmk} The additional $-i\ddz$ terms in
  Definition~\ref{stratadef} (or equivalently, the $-\tx\ddz$ term in
  the previous proposition) are needed to make the map $\sQ\to\sS$
  equivariant.  The $-\tx\ddz$ term also plays an important role in
  our study of the isomonodromy equations for flat vector bundles
  (c.f. \cite[Definition 2.12]{BrSa2}).
\end{rmk}

\begin{rmk}\label{stratarmk} Any trivialization $\phi$ and $x\in\B$
  (or equivalently, any point in $\B(\Aut(\sG)$) determines a
  stratum contained in $(\sG,\n)$.  By equivariance, it is enough to
  show this for $x\in\Ao$.  Let $r$ be the smallest critical number
  satisfying $[\n]_\phi - \tx \ddz \in \hfg_{x,r+}^\perp$.  If $\b\in
  (\hfg_{x, r} / \hfg_{x,r+})^\vee$ is the induced functional, then by
  Proposition~\ref{strataprop}, $(\sG,\n)$ contains $(x,r,\b)$ with
  respect to $\phi$.
\end{rmk}

In general, a nonfundamental stratum provides very little information
about a flat $G$-bundle.  However, we will see that all flat
$G$-bundles contain fundamental strata, and the depth of any such
stratum determines whether the flat $G$-bundle is regular singular,
and if not, how irregular it is.  Moreover, we will see that it always
possible to find a fundamental stratum for which $x \in \B$ is an
\emph{optimal point} in the sense of \cite[Section
6]{MP1}.

We now recall the definition of optimal points.  Fix an alcove $C
\subset \Ao$, and let $\Sigma_C$ be the collection of minimal affine
roots on $C$, that is, the set of affine roots $\psi$ such that for
all $x \in \bar{C}$, $0\le \psi(x) \le 1$. For any nonempty subset
$\Xi\subset\Sigma_C$, define a function $f_\Xi$ on $\bar{C}$ by
\begin{equation*}
  f_\Xi (x) = \min \{ \psi(x) \mid \psi \in \Xi\}.
\end{equation*}
Choose a point $x_\Xi\in\bar{C}$ at which $f_\Xi$ attains its maximum
value and $\widetilde{x_\Xi}\in\ft_\Q$.  We can further assume that
$x_\Xi\in\bB$.   The set of optimal points in $\bar{C}$ is given by
$\Opt_C=\{x_\Xi\}$.  It can be shown that $\Opt_C$ contains all the
vertices of $\bar{C}\cap\bB$.

\begin{exam} For $\SL(V)$, there is no ambiguity about the optimal
  points in a closed alcove; they can only be taken to be the
  barycenters of the simplices.  With our convention that optimal
  points are in the reduced building, the optimal points for $\GL(V)$
  are also uniquely determined.  Thus, optimal points here give rise to
  lattice chain filtrations.
\end{exam}

We will need to understand the effect of change of trivialization on
strata contained in $\sG$.  The following lemma establishes the
necessary calculus.

\begin{lem}\mbox{}\label{actlem}
\begin{enumerate}
 \item \label{act1}  If $n \in \hN$, $([\n]_{n \phi}
   - \widetilde{n x} \ddz) \in \Ad^*(n) ( [\n]_{ \phi } - \tx \ddz) +
   \hft_{0+}\ddz$.
\item\label{act2} If $X \in \hfu_\a \cap \hfg_{x, \ell}$,
then
\begin{equation*}
[\n]_{\exp(X) \phi} - \tx \ddz \in \Ad^*(\exp(X)) ([\n]_\phi - \tx \ddz) -
\ell X \ddz +
\hfg^\vee_{x, \ell+}.
\end{equation*}
\item\label{act3} If $p \in \hG_{x}$, then
$[\n]_{p \phi} -  \tx \ddz \in \Ad^*(p)([\n]_\phi - \tx \ddz)  +
\hfg^\vee_{x,0+}$.
\item\label{act4} If $p \in \hG_{x, \ell}$ for $\ell > 0$, then
$[\n]_{p \phi} -  \tx \ddz \in \Ad^*(p)([\n]_\phi - \tx \ddz)  +
\hfg^\vee_{x, \ell}$.

\end{enumerate}
\end{lem}
\begin{proof}

  By Lemma~\ref{bnorm}, $\widetilde{n x} \in \Ad(n) (\tx) - \tau(n)
  n^{-1} + \hft_{0+}$.  Dualizing, we obtain $\widetilde{n x} \ddz \in
  \Ad^*(n) (\tx \ddz) - (dn) n^{-1} + \hft_{0+}\ddz$.
  Part~\eqref{act1} follows by applying the gauge transformation
  formula and substituting the above expression into $[\n]_{n \phi} -
  \widetilde{n x} \ddz$.

Suppose that $X \in \hfu_\a \cap \hfg_{x, \ell}$, and write $u =
\exp(X)$.  In order to prove \eqref{act2}, it suffices to show that
\begin{equation*}
  (d u) u^{-1} \in \Ad^*(u) (\tx \ddz) - \tx \ddz + \ell X \ddz + \hfg_{x,
    \ell+}^\vee.
\end{equation*}
Recall that $\tau (X)+\ad (\tx) (X) \in \ell X + \hfg_{x, \ell+}$ by
Proposition~\ref{ev}.  Thus,
\begin{equation*}
  \begin{aligned}
(d u) u^{-1} &= \tau (X)\ddz \in -\ad (\tx) (X) \ddz  + \ell X \ddz + \hfg^\vee_{x, \ell+} \\
&= \left(\Ad(u) (\tx) - \tx \right) \ddz + \ell X \ddz +\hfg^\vee_{x, \ell+}.
\end{aligned}
\end{equation*}
The right hand side of this equation proves the desired result.

Observe that when $t \in T_\ell$, $(d t ) (t^{-1}) \in \ft_1 \ddz $
(resp. $\ft_\ell \ddz$ when $t \in T_\ell$ and $\ell > 0$).
Furthermore, $\widetilde{t x} = \tx$ for all $t \in T_\ell$, since
$T_\ell \subseteq T_0$.  Since $\hG_{x, \ell}$ is generated by the
root subgroups $\hU_\a \cap \hG_{x, \ell}$ and the congruence
subgroups $T_\ell \subset T$, statements \eqref{act3} and \eqref{act4}
now follow from \eqref{act2}.

\end{proof}

Recall that the treatment of formal flat vector bundles in
\cite{BrSa1} only involved strata coming from lattice chain
filtrations.  Moreover, the definition of containment there differs
from Definition~\ref{vbcontaindef}, and it is not equivalent for strata
of depth $0$.  However, we will need to know that certain results
from~\cite{BrSa1} remain true using the present definition of
containment and allowing arbitrary periodic filtrations.  The
following theorem generalizes \cite[Theorem 4.10]{BrSa1}; we relegate
the proof to Section~\ref{nonbaryproof} of the  appendix.

\begin{prop} \label{nonbary}The slope of a flat vector bundle $(U,\n)$
  is the minimum of the depths of the strata $(x,r,\b)$ contained in
  it.  If $(x,r,\b)$ is fundamental, then $r=\slope(U,\n)$; the
  converse is true if $r>0$.  In particular, $(U,\n)$ is regular
  singular if and only if it contains a stratum of depth $0$.
\end{prop}

\subsection{Regular and irregular singularity}

Recall that a formal flat $G$-bundle $(\sG, \n)$ is regular singular
if for all representations $V$, the associated flat vector bundle
$(V_\sG, \n_V)$ is regular singular.  Otherwise, it is called
irregular singular.  As one would hope, a flat $\GL_n$-bundle is
regular singular if and only if the corresponding flat vector bundle
is regular singular.  This is indeed true; see
Corollary~\ref{glslope}.

We next show that if a flat $G$-bundle contains a stratum of depth
$r$, then the same is true for each associated flat vector bundle.

\begin{prop}\label{rhoG} 
  Let $\rho : G \to \GL(V)$ be a representation for $G$.
\begin{enumerate}\item For any $x\in\B$, there is a uniquely determined $\rho_*(x) \in
  \B(\GL(\hV))$ that induces the same filtration on $\hV$ as $x$.
  Moreover, $\gl(\hV)_{\rho_*(x), r} \cap \rho (\hfg) =
  \rho(\hfg_{x,r})$ for all $r \in \R$.
\item If $T_\rho$ is a maximal torus in $\GL(V)$ containing $\rho(T)$,
  then  $\rho_*$ restricts to the map $\Ao\to
  \A(\widehat{T_\rho})$ determined by $\widetilde{\rho_*(x)}=\rho(\tx)$.
\end{enumerate}
\end{prop}
\begin{proof}

  Since $x$ determines a periodic filtration on $\hV$ and
  $\B(\GL(\hV))$ parameterizes such filtrations, there is a unique
  $\rho_*(x)\in\B(\GL(\hV))$ such that $\hV_{\rho_*(x),r}=\hV_{x,r}$
  for all $r$.  Next, choose a maximal torus $T_\rho \subset \GL(V)$
  such that $\rho (T) \subset T_\rho$. If $x\in\Ao$, let
  $y\in\A(\widehat{T_\rho})\subset\B(\GL(\hV))$ be the point determined by
  $\ty=\rho(\tx)$. To prove that $y=\rho_*(x)$, we must show that
  $\hV_{y,r}=\hV_{x,r}$ for all $r$.  It suffices to show that
  $\hV_{y}(r)=\hV_{x}(r)$ for all $r$.  By
  Proposition~\ref{ev}\eqref{ev1}, $\hV_x(r)$ is the $r$-eigenspace of
  the action of $\tau+\tx$ on $\hV$.  Since this operator on $\hV$
  coincides with $\tau+\rho(\tx)=\tau+\widetilde{\rho_*(x)}$, another
  application of Proposition~\ref{ev}\eqref{ev1} gives the desired
  equality of graded spaces.

  To prove the remaining statement, first note that by equivariance of
  Moy-Prasad filtrations, we may assume that $x\in\Ao$.  Observe that
  $(\tau + \ad(\rho(\tx)) \rho (X) = \rho ((\tau + \ad (\tx)) (X))$
  for all $X \in \hfg$.  Therefore, if $X_r \in \hfg_x(r),$ then $\rho
  (X_r) \in \gl(\hV)_{\rho_*(x)} (r)$.  By continuity, there exists $s
  \ge r$ such that $\hfg_{x, s} \subset
  \rho^{-1}(\gl(V)_{\rho_*(x),r+})$.  Repeated application of
  Proposition~\ref{ev}\eqref{ev3} shows that $\sum_{r\le j < s}
  \hfg_x(j)$ constitutes a full set of coset representatives for
  $\hfg_{x,r} / \hfg_{x, s}$.  Since $\rho(\hfg_x(j)) \subset
  \gl(V)_{\rho_*(x), r}$ by the work above, we deduce that $\rho
  (\hfg_{x, r}) \subset \gl(V)_{\rho_*(x),r}$.

  We now show that $X \in\gl(\hV)_{\rho(x), r} \cap \rho (\hfg)$
  implies $X \in \rho(\hfg_{x,r})$.  If this is false, let $r'<r$ be
  the largest critical number $s$ for which $\hfg_{x,s}$ intersects
  $\rho^{-1}(X)$.  Take $Y\in\hfg_{x, r'}$ with $\rho(Y)=X$.  By
  Proposition~\ref{ev}\eqref{ev3}, there exists $Y_{r'} \in \hfg_x
  (r')$ such that $Y\in Y_{r'}+\hfg_{x, r'+}$ The argument above shows
  that $\rho (Y_{r'}) \in X + \gl(\hV)_{\rho_*(x), r'+}$, and since
  $r' < r$, this implies that $\rho (Y_{r'}) \in \gl(\hV)_{\rho_*(x),
    r'+}$.  However, we also have $\rho(Y_{r'})\in
  \gl(\hV)_{\rho_*(x)} (r')$.  Since $\gl(\hV)_{\rho_*(x), r'+} \cap
  \gl(\hV)_{\rho_*(x)} (r') = \{0\}$, we deduce that $\rho(Y_{r'}) =
  0$.  This means that $\rho(Y-Y_{r'})=X$, but $Y-Y_{r'}\notin
  \hfg_{x,r'}$, contradicting the definition of $r'$.  It follows that
  $X \in \rho (\hfg_{x, r})$.  

\end{proof}

If $\b$ is a functional on $\hfg_{x,r}/\hfg_{x,r+}$ and $(V,\rho)$ is
a representation of $G$, we can use the proposition to define a
functional $\rho(\b)$ on $\gl(V)_{\rho_*(x),r}/\gl(V)_{\rho_*(x),r+}$
which is represented (with respect to the residue of the trace form)
by $\rho(\tb)\in\gl(V)^\vee_{\rho_*(x),-r}$.  The proposition also
shows that the functional is independent of the choice of
representative $\tb$.

\begin{cor}\label{rhostrat} Let $(V,\rho)$ be a representation of $G$.  If $(\sG,\n)$
  contains the stratum $(x,r,\b)$ with respect to $\phi$, then
  $(V_\sG,\n_V)$ contains the stratum $(\rho_*(x),r,\rho(\b))$ with
  respect to the induced trivialization $\phi_V$.
\end{cor}

\begin{proof} By Definition~\ref{stratadef}, $(\n_\phi - i \ddz -
  X_{\tb})(\hV_{x, i}) \subset \Om^1(\hV)_{x, (i-r)+}$ for all
  $i\in\R$.  Since is obvious that
  $X_{\rho(\tb)}=\rho(X_{\tb})\in\Om^1(\gl(\hV))$, the action of $(\n_\phi - i \ddz -
  X_{\tb})$ on $\hV$ is the same as that given by $\rho(\n_\phi - i \ddz -
  X_{\tb})=\n_{V,\phi}-i\ddz-X_{\rho(\tb)}$.  Applying
    Proposition~\ref{rhoG}, we obtain $(\n_\phi - i \ddz -
  X_{\tb})(\hV_{\rho_*(x), i}) \subset \Om^1(\hV)_{\rho_*(x), (i-r)+}$
  for all $i$.  By Definition~\ref{vbcontaindef}, $(\hV,\n_{V,\phi})$
  contains the stratum $(\rho_*(x),r,\rho(\b))$.
\end{proof}

Since the slope of a vector bundle is the minimum depth of a stratum
that it contains, we immediately obtain an upper bound on the slope of
$(V_\sG,\n_V)$.

\begin{cor}\label{upperbound} The slope of $(V_\sG,\n_V)$ is not greater than the
  minimum depth of a stratum contained in $(\sG,\n)$.
\end{cor}

Next, we show that if a flat $G$-bundle contains a stratum of depth
$0$, then it is regular singular.  We will see later that the converse
is also true.

 \begin{prop}\label{regsing}    Suppose that
   the flat $G$-bundle $(\sG, \n)$ contains a stratum $(x, 0, \b)$ of
   depth $0$ with respect to the trivialization $\phi$.  The following
   statements hold.
   \begin{enumerate}\item The flat $G$-bundle $(\sG, \n)$ is regular
     singular.
   \item Suppose $x\in\Ao$, and let $\Delta\subset\Ao$ be the open
     facet containing $x$.  Then, for any $y \in\bar{\Delta}$,
     $(\sG,\n)$ contains the stratum $(y, 0, \b^y)$ with respect to
     $\phi$, with $\b^y$ induced by $[\n]_\phi-\ty\ddz$.  In
     particular, this is true for $y$ in a minimal facet contained in
     $\bar{\Delta}$.
   \end{enumerate}
 \end{prop}

 \begin{proof} By equivariance, we may assume that $x\in\Ao$.  By
   Proposition~\ref{strataprop}, $[\n]_\phi-\tx\frac{dz}{z}\in
   \hfg^\perp_{x, 0+}$.  Since the same is true for $\tx\frac{dz}{z}$,
   $[\n]_\phi\in\hfg^\perp_{x, 0+}$ as well.  Take $y\in\bar{\Delta}$.
   Since $\hfg_{y+} \subset \hfg_{x+}$, we have
   $[\n]_\phi\in\hfg^\perp_{x+} \subset \hfg^\perp_{y+}$.
   Accordingly, $[\n]_\phi-\ty\frac{dz}{z}\in\hfg^\perp_{y+}$, and
   $(y,0,\b^y)$ is contained in $(\sG, \n)$.

   Now, suppose that $V$ is a finite dimensional representation for
   $G$. We have $[\n]_\phi = \iota_\tau([\n]_{\phi}) \ddz$ with
   $\iota_\tau([\n]_{\phi})\in \hfg_{x, 0}$.  Since $\hfg_{x, 0}$
   preserves the lattice $\hV_{x, 0} \subset \hV$,
   $\iota_\tau([\n]_{\phi})(\hV_{x,0})\subset \hV_{x,0}$, and it
   follows that $(V_\sG, \n_V)$ is regular singular.
  
 \end{proof}

If a flat $G$-bundle contains a fundamental stratum of positive depth,
then at least one of the associated vector bundles is irregular
singular.  In fact, we can be more specific.

\begin{prop}\label{adjchar}
  If $(\sG, \n)$ contains a fundamental stratum $(x, r, \b)$ of depth
  $r > 0$, then either the flat bundle $(\fg_{\sG},\n_\fg)$
  corresponding to the adjoint representation has slope $r$ or $G$ has
  a one dimensional representation $W$ such that $(W_\sG,\n_W)$ has
  slope $r$.
\end{prop}
\begin{proof}

  Suppose that $\slope(\fg_{\sG},\n_\fg)<r$.  We will show that $G$
  has a character $(W,\chi)$ for which $(W_\sG,\n_W)$ has slope $r$.

  As usual, we can assume that $x\in\Ao$.  Since $(x,r,\b)$ is
  fundamental, Proposition~\ref{nilpotent} implies that the coset
  $\iota_\tau([\n]_{\phi})+\hfg_{x, -r+}$ contains no nilpotent
  elements.  (We omit the $\tx$ term, since $r>0$.)  Let
  $Y\in\hfg_x(-r)$ be the homogeneous coset representative. By
  corollary~\ref{rhostrat}, $(\fg_{\sG},\n_\fg)$ contains the
  corresponding $\GL(\fg)$-stratum $(\ad_*(x),r,\ad(Y)\ddz)$.
  Proposition~\ref{nonbary} implies that this stratum is not
  fundamental, so the graded representative $\ad(Y)$ is nilpotent.
  This means that $Y=Y_1+Y_2$ with $Y_1\in\hfz(-r)$ nonzero and $Y_2$
  a nilpotent element of $\fg_x(-r)\cap[\hfg,\hfg]$.  Note that this
  already implies that $r\in\Z_{>0}$, since homogeneous element of the
  center have integral degrees.

  Since the connected center $Z^0\cong G/[G,G]$ is a torus, there
  exists a character $(W,\chi)$ of $G$, vanishing on $[G,G]$ such that
  $\chi(Y_1)\in z^{-r}k^*$.  The corresponding stratum contained in
  $(W_\sG,\n_W)$ is $(\chi_*(x),r,\chi(Y_1)\ddz)$, which is evidently
  fundamental.
 \end{proof}

 Flat line bundles have integral slope, so we obtain the following
 corollary.

 \begin{cor} If $(\sG,\n)$ contains a fundamental stratum of
   nonintegral depth $r$, then $\slope(\fg_\sG,\n_\fg)=r$.
\end{cor}

\subsection{Proofs of the main theorems}\label{proofs}

Recall that $\Psi_C$ denotes the set of optimal points in a given
closed chamber $\bar{C}\subset\Ao$.

\begin{lem}\label{unstable}  Suppose that $(\sG, \n)$ contains a stratum $(x,
  r, \b)$ of depth $r > 0$
\begin{enumerate}
\item If $r$ is not a critical number for $x$, then $(\sG, \n)$
  contains a stratum of the form $(x, s, \b')$ where $s < r$ is a
  critical number.
\item If $r$ is a critical number and $(x, r, \b)$ is not fundamental,
  then $(\sG, \n)$ contains a stratum $(y, s, \b')$ with $y\in\Psi_C$
  and $s < r$.
\end{enumerate}
\end{lem}

\begin{proof}

  By Lemma~\ref{strataequivariance}, we may assume without loss of
  generality that $x\in\bar{C}$.  First, suppose that $(\sG, \n)$
  contains $(x, r, \b)$ with respect to the trivialization $\phi$ and
  that $r\notin\Crit_x$.  Let $s$ be the greatest critical number for
  $x$ less than $r$.  We claim that $[\n]_\phi \in \hfg^\vee_{x, -s}$.
  Write $[\n]_\phi = (Y_0+\sum_{\a \in \Phi} Y_\a) \ddz$, where
  $Y_0\in\hft$ and $Y_\a\in\hfu_\a$. If $[\n]_\phi \notin \hfg^\vee_{x,
    -s}$, then $Y_\g \ddz \notin \hfg^\vee_{x, -s}$ for some $\g \in
  \Phi\cup\{0\}$.  Since $Y_\g \ddz \in \hfg^\vee_{x, -r}$, we may
  take $s' = \min\{t \in (s, r] \mid Y_\g \ddz \in \hfg^\vee_{x,
    -t}\}>s$.  However, since $Y_\g \ddz$ has nonzero image in
  $\hfg^\vee_{x, -s'}/\hfg^\vee_{x, -s'+}$, $s'$ is a critical point
  with $s'<r$.  This contradicts the assumption that $s$ is the
  greatest critical point less than $r$.  We now apply
  Proposition~\ref{strataprop} to see that $(\sG,\n)$ contains the
  stratum $(x,s,\b')$ with respect to $\phi$, where $\b'$ is induced
  by $[\n]_\phi-\tx\ddz$.

  We now assume that $r$ is a critical number for $x$ and $(x,r,\b)$
  is not fundamental.  By \cite[Proposition 6.3]{MP1}, there exists $p
  \in \hG_x$ and $y \in \Opt_C$ such that for some $s < r$, $\Ad^*(p)
  ([\n]_\phi - \tx \ddz) + \hfg^\vee_{x, -r+} \subset \hfg^\vee_{y,
    -s}$.  Lemma~\ref{actlem}~\eqref{act3} implies that
  $\Ad^*(p)([\n]_\phi - \tx \ddz) + \hfg^\vee_{x, -r+} = [\n]_{p \phi}
  - \tx \ddz + \hfg^\vee_{x, -r+}$.  Since $[\n]_{p \phi} - \tx \ddz
  \in \hfg^\vee_{y, -s}$ and $(\tx -\ty)\ddz \in \hfg^\vee_{y, 0}$,
  $[\n]_{p \phi} - \ty \ddz \in \hfg^\vee_{y, -s}$.  Letting $\b'$ be
  the functional induced by $[\n]_{p \phi} - \ty \ddz$, it follows
  that $(\sG, \n)$ contains the stratum $(y, s, \b')$ with respect to
  the trivialization $p \phi$.
\end{proof}

\begin{prop} \label{minr} Suppose that $(\sG, \n)$ contains a
  fundamental stratum $(x, r, \b)$ of depth $r > 0$.  If $(y,s,\b')$
  is another stratum contained in $(\sG, \n)$, then $s\ge r$.
  Moreover, if $s=r$, this stratum is fundamental.
\end{prop}
\begin{proof} Recall that any two points in the building lie in a
  common apartment.  Hence, by equivariance, we may assume without
  loss of generality that the two strata are contained in $(\sG, \n)$
  with respect to the same trivialization and that $x,y\in\Ao$.

Suppose that $s<r$.  We can apply \cite[Proposition~6.4]{MP1} to
obtain
$\hfg^\vee_{y, -s} \cap (\tb + \hfg^\vee_{x, -r+})= \emptyset$.  By
Proposition~\ref{strataprop}, there exists $\omega\in\hfg^\vee_{x, -r+}$
and $\omega'\in\hfg^\vee_{y, -s+}$ such that $\tb=[\n]_\phi-\tx\ddz+\omega$
and $\tb'=[\n]_\phi-\ty\ddz+\omega'$.  Since
$(\ty-\tx)\ddz\in\hfg^\vee_{y, 0}$, it follows that
$\tb'-\omega'+(\ty-\tx)\ddz=\tb-\omega\in\hfg^\vee_{y, -s} \cap (\tb +
\hfg^\vee_{x, -r+})$. This contradiction implies that $s \ge
r$.

Now, suppose that $s = r$.  If $(y, r, \b')$ is not fundamental, then
Lemma~\ref{unstable} states that $(\sG, \n)$ contains a stratum
of depth $s'$ strictly less than $r$.  This contradicts the conclusion
of the previous paragraph.

\end{proof}

\begin{cor}\label{minrcor} All fundamental strata contained in $(\sG,\n)$ have the
  same depth.
\end{cor}
\begin{proof} Suppose $(\sG,\n)$ contains two fundamental strata with
  different depths $r$ and $s$, say with $r>s$.  Then $r>0$, so the
  proposition gives the contradiction $s\ge r$.
\end{proof}

\begin{proof}[Proof of Theorem~\ref{MP}]
  We only need to consider $\hG$-strata in the proof as the statements
  about $\sG$-strata are immediate consequences.  We first prove
  that every flat $G$-bundle $(\sG, \n)$ contains a fundamental
  stratum (resp. a stratum of depth $0$) if it is irregular singular
  (resp. regular singular).  Fix a chamber $C\subset\Ao$, and let
  $\Opt_C$ be the set of optimal points in $\bar{C}$.  Consider the
  set of positive real numbers $s$ such that $(\sG, \n)$ contains a
  stratum $(x,s,\b)$ with $x\in\Opt_C$.  This set is clearly nonempty,
  as by Remark~\ref{stratarmk}, any choice of trivialization and
  optimal point determines a stratum for $(\sG, \n)$.  Since $\Opt_C$
  is finite, it attains its lower bound $r>0$.

  Suppose that $(x,r, \b)$ is a stratum contained in $(\sG, \n)$ such
  that $x \in \Opt_C$ and $r$ is the lower bound described above.  We
  write $\phi$ for the associated trivialization of $\sG$.  Assume
  that this stratum is not fundamental.  Note that
  Proposition~\ref{adjchar} guarantees that this is the case when
  $(\sG,\n)$ is regular singular.  Lemma~\ref{unstable} implies that
  $(\sG,\n)$ contains a stratum of the form $(y, s, \b')$, where
  $y\in\Opt_C$ and $s<r$.  (If $r$ is not a critical point, one can
  take $y=x$.)  By minimality of $r$, we must have $s=0$, and it
  follows from Proposition~\ref{regsing} that $(\sG,\n)$ is regular
  singular.  Thus, $(x,r,\b)$ is fundamental if $(\sG,\n)$ irregular
  singular, and $(\sG,\n)$ contains a stratum of depth $0$ if it is
  regular singular.  Statements~\eqref{MP1} and ~\eqref{MP3} now
  follow immediately from Proposition~\ref{minr} and
  Corollary~\ref{minrcor}.

  Next, suppose that $r = 0$.  It remains to show that $(\sG, \n)$
  contains a fundamental stratum.  We have shown that $(\sG, \n)$
  contains a stratum of the form $(x, 0, \b)$ with respect to a
  trivialization $\phi$, and by Proposition~\ref{regsing} we may
  assume that $x$ is in a minimal facet $\Delta\subset\Ao$.  If this
  stratum is nonfundamental, then, using the notation of
  Proposition~\ref{r0stratum}, $d\theta_x^* (\b) \in \fh^\vee_x$ is
  unstable.  In particular, there exists a Borel subgroup $B_\b
  \subset H_x$ such that $d\theta_x^*(\b) \in (\fb_\b)^\perp$. Choose
  $h \in H_x$ such that $hB_\b h^{-1}$ is the standard Borel subgroup
  $B_x \subset H_x$ containing $T$ coming from a choice of positive
  roots for $G$, and let $m \in G_x$ be a lift of $\theta_x(h)$.  By
  part \eqref{act3} of Lemma~\ref{actlem}, $(\sG, \n)$ contains the
  stratum $m \cdot (x, 0, \b)$.  Thus, we may assume without loss of
  generality that the stratum $(x,0,\b)$ satisfies $d \theta_x^*(\b)
  \in (\fb_x)^\perp$, where $\fb_x = \Lie(B_x)$.

  Let $\de \in \ft$ be the element corresponding to the half sum of
  positive coroots of $H_x$.  We define $x_\e \in \Ao$ via $\tx_\e=\tx
  + \e \de$.  For small $\e > 0$, we will show that $\tbo +
  \hfg^\vee_{x+} \subset \hfg^\vee_{x_\e+}$.  (Recall that $\tbo$ is
  the homogeneous representative for $\b$.)  First, since $x$ is in a
  minimal facet, $x + \e \de$ must be contained in a chamber $C$ with
  $x \in \bar{C}$ for $\e$ sufficiently small. (Note that $x \notin
  C$).  Therefore, $\hfg_{x_\e} \subset \hfg_{x}$, so taking
  annihilators gives $\hfg^\vee_{x+} \subset \hfg^\vee_{x_\e+}$. An
  elementary calculation, using the fact that $\ad^* (\e \de)$ has
  strictly positive eigenvalues on $\fb_x^\perp$, shows that $(d
  \theta_x^*)^{-1} (\fb_x^\perp) \subset \hfg^\vee_{x_\e+} +
  \hfg^\vee_{x+}$.  We conclude that $\tbo \in \hfg_{x_\e+}$ since $d
  \theta_x^*(\b) \in (\fb)^\perp$.  This proves the assertion above.

Write $\tb_\e = \tbo - \e \de \ddz \in \hfg_{x_\e}^\vee$ and let
$\b_\e$ be the corresponding element of $(\hfg_{x_\e}/
\hfg_{x_\e+})^\vee$.  We deduce that $(\sG, \n)$ contains the stratum
$(x_\e, 0, \b_\e)$ with respect to the trivialization $\phi$. By
Proposition~\ref{regsing}, $(\sG, \n)$ contains the stratum $(y, 0,
\b^y)$for any $y\in\bar{C}$, where $\b^y$ is determined by
$[\n]_{m\phi} - \ty \ddz$.

Finally, since $[\n]_{\phi} - \tx \ddz \in \tbo + \hfg^\vee_{x+}
\subset \hfg^\vee_{x_\e+}$, it follows that $[\n]_{\phi} - \ty \ddz +
\hfg^\vee_{x_\e+} \in (\tx - \ty) \ddz + \hfg^\vee_{x_\e+}$; as
$(\tx-\ty)\in\hfg^\vee(0)$, the proof of Proposition~\ref{nilpotent}
shows that this coset is semistable if $\ty\ne\tx$.  Therefore,
$[\n]_{\phi} - \ty \ddz + \hfg^\vee_{y, 0+}$ has the same property, so
$(y,0, \b')$ is fundamental.  In particular, this is true for any
optimal point in $\bar{C}$ besides $x$, for example, any other vertex
of $\bar{C}\cap\bB$.

\end{proof}

\begin{rmk} If $(\sG,\n)$ contains a nonfundamental stratum $(x,0,\b)$
  with $x$ in a minimal facet, then the proof above, replacing $\de$
  with $w\de$ for $w$ in the Weyl group of $H_x$, gives a construction
  of a fundamental stratum contained in $(\sG,\n)$ at every $y\ne x$
  in the closed star of $x$ in $\Ao$.  Moreover, this stratum induces
  a fundamental stratum in $\fg_\sG$ as long as $y$ is not in the same facet
  as $x$.
\end{rmk}

We now turn to the proof of Theorem~\ref{associateness}, which states
that strata of the same depth contained in a flat $G$-bundle are
associates of each other.  First, we supply the proof of
Proposition~\ref{deltaprop}, which was needed for
Definition~\ref{assdef}.

\begin{proof}[Proof of Proposition~\ref{deltaprop}]
Choose $g\in\hG$ such that $gx,gy\in\Ao$.
(For example, let $\A$ be an apartment containing $x$ and $y$, and
take $g$ such that $g \A = \Ao$.)  We now set
$\de_{x,y} = \Ad^*(g^{-1}) ((\widetilde{gy} - \widetilde{gx}) \ddz)$.
This, of course, depends on the choice of $g$, but the defining
property will be satisfied independently of the choice. 

In order to show~\eqref{delta}, one may easily reduce to the special
case of $x,y\in\Ao$ and $g \in \hG_x \cap \hG_y$. Here, we will set
$\de_{x, y} = (\tx-\ty)\ddz$.  In other words, we must show that
whenever $g \in \hG_x \cap \hG_y$,
\begin{equation}\label{assline}
\Ad^*(g)(\de_{x, y}) \in \de_{x, y} + \hfg^\vee_{x, 0+}+ \hfg^\vee_{y,0+}.
\end{equation}
Using the notation introduced after Lemma~\ref{strataequivariance}, observe
that $\de_{x, y} = \dAo{y} - \dAo{x}$.  However, by \eqref{dA}, $g
\dAo{x} g^{-1} \in \dAo{x} + \hfg^\vee_{x, 0+}$ and $g \dAo{y} g^{-1}
\in \dAo{y} + \hfg^\vee_{y, 0+}$.  Subtracting, we obtain
\eqref{assline}.

\end{proof}

\begin{proof}[Proof of Theorem~\ref{associateness}]
  Since conjugate strata are associates of each other, we may assume
  without loss of generality that $\phi'=\phi$.  Applying
  Lemma~\ref{strataequivariance} and the fact that we can find
  $h\in\hG$ such that $hx,hy\in \Ao$, we may further assume that
  $x,y\in\Ao$.

  We now verify~\eqref{asseq} in this situation with $g=1$.  By
  Proposition~\ref{strataprop}, $\tb \in [\n]_\phi - \tx \ddz +
  \hfg^\vee_{x, -r+}$ and $\tb' \in [\n]_\phi - \ty \ddz +
  \hfg^\vee_{y, -r+}$.  Taking $\de_{x, y} = (\tx-\ty)\ddz$, it is
  immediate that $[\n]_\phi - \tx \ddz$ lies in the intersection
  $\left(\tb + \hfg^\vee_{x, -r+}\right) \cap \left(\tb' - \de_{ x, y}
    + \hfg^\vee_{y, -r+}\right)$.

\end{proof}

\begin{proof}[Proof of Theorem~\ref{slopethm}] By
  Corollary~\ref{upperbound}, $\slope(V_\sG,\n_V)$ is at most the
  minimum of the depths of the strata contained in $(\sG,\n)$.  The
  equivalence of the four characterizations and the positivity
  statement follow from Theorem~\ref{MP} and
  Proposition~\ref{adjchar}.  The slope is rational, since slopes of
  flat vector bundles are rational.
\end{proof}

Finally, we turn to the proof of Proposition~\ref{FGslope}.  Let $E$ is
a degree $e$ field extension of $F$, and fix a generator $u_E$
satisfying $u_E^e=z$.  We let $\pi_E : \Spec (E) \to \Spec(F)$ be the
associated map of spectra.  We let $\B(E)$ be the building for $G(E)$
and denote the apartment corresponding to $T(E)$ by $\Ao(E)$.  The
pullback of the flat $G$-bundle $(\sG,\n)$ to $\Spec(E)$ will be
denoted by $(\pi_E^*\sG, \pi_E^*\n)$.  We will suppress the subscripts
when the field $E$ is clear from context.

\begin{proof}[Proof of Lemma~\ref{slopepullback}]
  Set $e=[E:F]$.  Take $x\in\Ao$, and let $ex\in\Ao(E)$ be the point
  corresponding to $e\tx$.  If $V$ is any representation of $G$, it is
  easily checked that $\hV_{x}(r)=V(E)_{ex}(er)\cap \hV$ and
  $\hV_{x,r}=V(E)_{ex,er}\cap \hV$.  Indeed, it suffices to check the
  statement about gradings.  Setting $\tau_E=u\frac{d}{du}=e\tau$,
  Proposition~\ref{ev}\eqref{ev1} implies that $V(E)_{ex}(er)\cap\hV$
  consists of those elements $v\in\hV$ such that
  $erv=(\tau_E+\widetilde{ex})(v)=e(\tau+\tx)(v)$, which is precisely
  $\hV_x(r)$.  We can now define a pullback map on strata based at
  points in $\Ao$: $\pi^*(x, r, \b) = (e x, e r, \b')$; here, $\b'$ is
  the functional on $\fg(E)_{e x, e r} / \fg(E)_{e x, e r+}$
  determined by the representative $e\tb$ coset in $\hfg^\vee_{x, -r}
  \subset \fg(E)^\vee_{e x, -e r}$.  Moreover, since $\b$ and $\b'$
  have graded representatives differing by a factor of $e$,
  Remark~\ref{fundrmk} implies that $\pi^*(x, r, \b)$ is fundamental
  if and only if $(x, r, \b)$ is.

  Since $(\sG,\n)$ has slope $r$, there exists a fundamental stratum
  $(x,r,\b)$ contained in $(\sG,\n)$ with respect to a trivialization
  $\phi$.  By equivariance, we may assume that $x\in\Ao$.  Applying
  Proposition~\ref{strataprop}, the functional $\b$ is determined by
  $[\n]_\phi-\tx\ddz\in\hfg^\vee_{x,-r}\subset\fg(E)^\vee_{ex,-er}$.
  However, $e[\pi^*\n]_\phi-\widetilde{ex}\ddu=[\n]_\phi-\tx\ddz$, so
  the same proposition shows that $(\pi^*\sG, \pi^*\n)$ contains the
  fundamental stratum $(ex,er,\b')$.  It follows that the pullback
  bundle has slope $er$.
\end{proof}

\begin{proof}[Proof of Proposition~\ref{FGslope}] Set $e=[E:F]$.  The
  condition involving \eqref{stdform} is equivalent to the statement
  that $(\pi^*\sG, \pi^*\n)$ contains a fundamental stratum $(o, n,
  \b)$ based at the origin of $\Ao(E)\subset\B(E)$.  It must then have
  slope $n=re$.  By Lemma~\ref{slopepullback}, its slope is also equal
  to $e\slope(\sG,\n)$, so the original flat $G$-bundle has slope $r$.
  On the other hand, by \cite[Theorem 9.5]{BaVa1}, there exists an
  algebraic field extension $E/F$ and a trivialization $\phi$ for
  $\pi^*\sG$ such that $(\pi^*[\n])_\phi$ has nonnilpotent leading
  term with respect to powers of $u$, say in degree $-n$.  Again, this
  means that the pullback bundle has slope $n$.  If the original
  bundle has slope $r$, then the same lemma implies that $r=n/e$ as
  desired.
\end{proof}

\section{Examples}

In this section, we provide some examples to illustrate the theory.
In each example, we write down the matrix for the flat $G$-bundle
$(\sG,\n)$ with respect to a fixed trivialization $\phi$, which will be
omitted from the notation.

\begin{exam} Let $m$ be a nonnegative integer, and set
  $[\n]=(\sum_{i\ge-m} X_i z^i)\ddz$, where $X_i\in\fg$ and $X_{-m}\ne
  0$.  Then, $(\sG,\n)$ contains the stratum $(o,m,\b^o)$, where
  $o\in\Ao$ is the origin and $\b^o$ is induced by $X_m\ddz$, so
  $\slope(\n)\le m$.  If $m>0$, then this stratum is fundamental at
  the origin in $\Ao$ if and only if $X_{-m}$ is not nilpotent, in
  which case, $\n$ has slope $m$.  If we assume that $X_{-m}$ is
  contained in a Borel subalgebra $\fb\supset\ft$ (which we can
  accomplish by a constant change of gauge), then $(x,m,\b^x)$ is
  contained in $(\sG,\n)$ for all $x\in\Ao$ if and only if
  $X_{-m}\in\ft$.  If $m=0$, $\slope(\n)=0$ for any $X_0$ while the
  stratum at $o$ is fundamental if and only if $X_0$ is nonnilpotent.
  If we assume that $X_0\in\fb$, then $(x,m,\b^x)$ is fundamental for
  all $x\in\Ao$ precisely when $X_0\in\ft-\ft_\R$.
\end{exam}

\begin{exam}\label{Cox} Suppose that $\fg$ has connected Dynkin diagram.  Let
  $\a_1,\dots,\a_n$ be a set of simple roots, and let $\a_0$ be the
  highest root.  Let $y_{-i}\in\fu_{-\a_i}$, $y_{\a_0}\in\fu_{\a_0}$
  be a collection of nonzero root vectors, and set
  $X=(z^{-1}y_{\a_0}+\sum_{i=1}^n y_{-i})$.  Fix $m\in\Z_{\ge 0}$, and
  let $[\n]=Xz^{-m}\ddz$.  As explained in the previous example, this
  flat $G$-bundle contains the stratum $(o,m+1,\b^o)$, but the stratum
  is not fundamental, since it is induced by the nilpotent element
  $z^{-m-1} y_{\a_0}$.  However, one readily checks that
  $X\in\hfg_{x}(-1/h)$, where $h$ is the Coxeter number and $x$ is any
  $\fz_\R$-translate of the barycenter corresponding to the standard
  Iwahori subgroup.  Since $X$ is regular semisimple, it follows that
  $(\sG,\n)$ contains a fundamental stratum based at $x$ of depth
  $m+\frac{1}{h}$, which is accordingly the slope.  Note that when
  $m=0$, this is the rigid flat $G$-bundle described by Katz for
  $\GL_n$ and Frenkel-Gross in general~\cite{FrGr}.

  In fact, no other points in $\Ao$ support a fundamental stratum
  contained in $\n$ with respect to $\phi$.  To see this, note that
  $X\in\hfg_{y,-1/h}$ implies that $\a_i(y)\le \frac{1}{h}$ for $1\le
  i\le n$ and $\a_0(y)\ge \frac{h-1}{h}$.  Using the fact that $\a_0$
  has height $h-1$ and adding the first inequalities appropriately, we
  get $\a_0(y)\le \frac{h-1}{h}$, so $\a_0(y)=\frac{h-1}{h}$.  This
  immediately gives $\a_i(y)=\frac{1}{h}$, so $y$ lies over the vertex
  in the reduced building corresponding to the standard Iwahori
  subgroup.

  A variation on this construction gives flat $G$-bundles with slope
  $m+\frac{h-1}{h}$: define $[\n']=X'z^{-m}\ddz$, where
  $X'=(z^{-1}y_{-\a_0}+\sum_{i=1}^n y_{i})$ and $y_{i}\in\fu_{\a_i}$,
  $y_{-\a_0}\in\fu_{-\a_0}$ are nonzero root vectors.

\end{exam}

For $\SL_n$, the previous example gives flat $\SL_n$-bundles of slope
$m+1/n$ and $m+(n-1)/n$.  The next example constructs flat
$\SL_n$-bundles with slopes that are integer translates of $1/(n-1)$.
Here, $\n$ supports fundamental strata on a line in $\Ao$.  For
clarity, we take $n=3$.

\begin{exam}\label{slnexample} Let $G=\SL_3(k)$.  Set $X=z^{-1}e_{12}+e_{21}\in\hfg$,
  and consider the flat $G$-bundle $[\n]=X z^{-m}\ddz$.  We show that
  $\slope(\n)=m+\frac{1}{2}$.  It suffices to find $x\in\Ao$ for which
  the regular semisimple matrix $X$ lies in $\hfg_{x}(-\frac{1}{2})$.
  Writing $\tx=\diag(x_1,x_2,x_3)$ with $x_1+x_2+x_3=0$, this is
  equivalent to $x_1-x_{2}=1/2$.  This is the line connecting the
  barycenters of the faces of the fundamental alcove where the last
  simple root $\a_{2}$ (with respect to the usual order) vanishes and
  the highest root $\a_{0}$ equals $1$.  It is easy to check that $x$
  supports a stratum of depth $m+\frac{1}{2}$ (with respect to $\phi$)
  precisely for $x$ on this line.  Note that this line does not
  contain a barycenter of an alcove; if it did, the slope of $\n$
  would have to be a multiple of $1/3$.

  A more intuitive explanation is obtained by looking at lattice
  chains.  For any $s \in \Z$, write $s = 3q+j$ with $0 \le j <3.$
  Now, define $L^s$ be the lattice with $\fo$-basis $\{z^{q}e_i\mid
  i\le 3-j\}\cup\{ z^{q+1}e_i\mid i> 3-j\}$.  The lattice chain
  $\cL=(L^s)$ corresponds to the fundamental alcove while the period
  $2$ lattice chains $\cL_j=(L^s\mid s\not\equiv j\mod 3)$ correspond
  to the faces.  In this terminology, $\cL_1$ and $\cL_0$ correspond
  to the faces $\a_{2}=0$ and $\a_0=1$ respectively. Note that
  $X(L^s)\subset L^{s-1}-L^s$ unless $s\equiv 1$, in which case
  $X(L^s)\subset L^{s-2}-L^{s-1}$.  Accordingly, the depth of the
  stratum contained in $\n$ induced by the lattice filtrations for
  $\cL$ is $m+2/3$ while for $\cL_j$, it is $m+1/2$ if $j= 0,1$ and
  $m+1$ if $j=2$.

  A similar analysis shows that the flat $\SL_n(k)$-bundle $[\n]=X
  z^{-m}\ddz$ with $X=z^{-1}e_{1(n-1)}+\sum_{i=1}^{n-2}e_{(i+1)i}$ has
  slope $m+\frac{1}{n-1}$.

\end{exam}

The following example shows that one can have a flat $G$-bundle that
supports a fundamental stratum only at a single point in $\bB$, which,
unlike Example~\ref{Cox}, is not in an alcove.

\begin{exam} Let $G=\Sp_{4}(k)$, with the form defined by $\langle
  e_i,e_{j+2}\rangle=\de_{ij}=-\langle e_{i+2},e_{j}\rangle$ and
  $\langle e_i,e_{j}\rangle=0$ for $1\le i,j\le 2$.  Set
  $Y=z^{-1}(e_{13}-e_{24})+e_{31}+e_{42}$, and let $\n$ be the flat
  $\Sp_4(k)$-bundle defined by $[\n]=Y z^{-m}\ddz$. Here, $\n$ is a
  connection of slope $m+\frac{1}{2}$, and there is a unique point in
  the standard apartment supporting a fundamental stratum for $\phi$
  with this depth.  Indeed, setting $\tx=(x_1,x_2,-x_1,-x_2)$, the $4$
  inequalities $-2x_i\ge -1/2$ and $2x_i-1\ge -1/2$ immediately give
  $x_1=x_2=1/4$.  This point is the barycenter of the edge of the
  fundamental alcove where the short simple root $\a_1$ vanishes.
  Since $Y$ is regular semisimple and is graded with respect to this
  filtration, the corresponding stratum at $x$ is fundamental.

  One can give a lattice-theoretic interpretation here as well.
  Parahoric subgroups of $\Sp_{2n}(k)$ are stabilizers of
  \emph{symplectic} lattice chains: lattice chains which are closed
  under homothety and duality with respect to the symplectic
  form~\cite{Sa00}.  For $n=2$, one period of the lattice chain $\cL$
  stabilized by the standard Iwahori subgroup is \begin{multline*}
    L^0=\fo^4\supset L^1=\spa\{e_1,e_2,z e_3,e_4\}\supset\\
    L^2=\spa\{e_1,e_2,z e_3,z e_4\}\supset L^3=\spa\{e_1,z e_2,z e_3,z
    e_4\}.
\end{multline*}
The edges are obtained from $\cL_j$ for $0\le j\le 2$, the symplectic
lattice chain generated by $L_i$, $0\le i\le 2$, $i\ne j$.  In
particular, $\a_1=0$, $\a_2=0$, and $\a_0=1$ (with $\a_2$ the long
simple root and $\a_0$ the highest root) correspond to $j$ equal to
$1$, $2$, and $0$ respectively.  We have $Y(L^{4q+3})\subset
L^{4q}-L^{4q+1}$ and $Y(L^j)\subset L^{j-2}-L^{j-1}$ otherwise. Thus,
$Y$ shifts $\cL^1$ by $-1$, $\cL^0$ and $\cL^2$ by $-2$, and $\cL$ by
$-3$.  The depth of the corresponding strata are $m+1/2$, $m+1$, and
$m+3/4$.  (Note that for partial symplectic lattice chains, one
does not obtain the depth by dividing the magnitude of the shift by
the period.)

\end{exam}

The previous three examples have the special property that they
contain fundamental strata whose graded representative in $\hfg^\vee$
has a maximal \emph{nonsplit} torus in $\hG$ as its connected
stabilizer under the coadjoint action.  In the case of $\GL_n$, these
\emph{regular strata} have been the key ingredient in constructing
smooth symplectic and Poisson moduli spaces of connections and have
allowed the realization of the isomonodromy equations as an integrable
system~\cite{BrSa1,BrSa2}.  Regular strata and flat $G$-bundles
containing them for reductive $G$ are studied in detail
in~\cite{BrSa5}.

\appendix

\section{Complements on flat vector bundles}\label{glappendix}

The authors have studied flat vector bundles using $\GL_n$-strata in
previous work~\cite{BrSa1}, and these results are cited frequently in
this paper.  However, there are certain differences between the set-up
in~\cite{BrSa1} and our present approach to flat vector bundles.  For
example, only lattice chain filtrations were considered
in~\cite{BrSa1}, and the notation for strata was given in terms of
parahoric subgroups instead of points in the building.  Furthermore,
the definition of containment of a stratum in a flat vector bundle
(\cite[Definition 4.1]{BrSa1}) is different from
Definition~\ref{vbcontaindef} and is only equivalent for strata of
depth greater than $0$.  Here, we provide the necessary explanations.
We also show that our theory gives the same results for the equivalent
concepts of flat rank $n$ vector bundles and flat $\GL_n$-bundles.
\subsection{Definitions of containment}

We begin with a general proposition about stratum containment.

\begin{prop}\label{containops} The subset of $\Rep(G)$ satisfying \eqref{stratadefeq} for
  the flat $G$-bundle $(\sG,\n)$ and the stratum $(x,r,\b)$ is closed
  under taking subrepresentations, duals, direct sums, tensor
  products, and homomorphism spaces.  Moreover, it always contains
  the trivial representation.
\end{prop}
\begin{proof} The fact that this set is closed under
  subrepresentations and direct sums is trivial.  Let $U$ and $W$ be
  two representations satisfying~\eqref{stratadefeq}.  Using the fact
  that $U^\vee=\Hom(U,k)$ and $U\otimes W\cong\Hom(U^\vee,W)$, it
  suffices to check ~\eqref{stratadefeq} for the trivial
  representation $k$ and for $\Hom(U,W)$.  For $\hat{k}=F$, each
  Moy-Prasad filtration is the usual one, so we need only check that
  the operator $\tau-i$ strictly increases the valuation for Laurent
  series of valuation at least $i$.  This follows since, for any
  $a\in\fo$ and $k\ge i$, $(\tau -i)(az^k)=(k-i)az^k+\frac{da}{dz}
  z^{k+1}$ has valuation at least $i+1$.

Next, take $f\in\widehat{\Hom(U,W)}_{x,i}=\Hom(\hU,\hat{W})_{x,i}$, so $f(U_{x,s})\subset W_{x,s+i}$ for
all $s$.  Recalling that $\n_\phi$ acts on $f$ via $\n_\phi\circ
f-f\circ \n_\phi$ and similarly for $X_{\tb}$, we compute:
\begin{equation*}
\begin{split}
 [(\n_\phi-i\ddz-X_{\tb})&(f)](U_{x,s})\\&=(\n_\phi-(i+s)\ddz-X_{\tb})(f(U_{x,s}))-f((\n_\phi-s\ddz-X_{\tb})(U_{x,s}))\\
&\subset (\n_\phi-(i+s)\ddz-X_{\tb})(W_{x,s+i})-f(U_{x,(s-r)+})\ddz\\
&\subset \Om^1(\hatW)_{x,(s+i-r)+}.
\end{split}
\end{equation*}
Hence, $(\n_\phi-i\ddz-X_{\tb})(\widehat{\Hom(U,W)}_{x,i})\subset
\Om^1(\widehat{\Hom(U,W)})_{x,(i-r)+}$ as desired.
\end{proof}

\begin{cor}\label{vbgl}  The $\GL_n$-stratum $(x,r,\b)$ is contained in the flat
  $\GL_n$-bundle $(\sG,\n)$ if and only if it is contained in the
  associated vector bundle for the standard representation.
\end{cor}
\begin{proof}
  Recall that $(x,r,\b)$ is contained in $(\sG,\n)$ if and only if
  \eqref{stratadefeq} holds for $W\in\Rep(\GL(V))$ while it is
  contained in $(\sG_V,\n_V)$ if the same equation holds for the
  standard representation $V$.  It thus suffices to show that if
  \eqref{stratadefeq} holds for the standard representation $V$, then
  it holds for any representation.  Since any representation of
  $\GL(V)$ is obtained from $V$ via some combination of the operations
  in the previous proposition, the result follows.
\end{proof}

In~\cite{BrSa1}, a $U$-stratum was defined as a triple $(P,s,\b)$ with
$P\subset\GL(U)$ a parahoric subgroup, $s$ a nonnegative integers, and
$\b$ a functional on the quotient of consecutive congruent subalgebras
$\b\in(\fP^s/\fP^{s+1})^\vee$.  The congruent subalgebras are defined
in terms of a lattice chain $\cL$ (say of period $e$) satisfying
$\Stab(\cL)=P$.  (This lattice chain is unique up to translation of
the indices by an integer.)  As explained in Section~\ref{filtF},
$\cL$ gives rise to uniform $\R$-filtrations $\{U^\cL_r\}$ and
$\{\gl(U)^\cL_r\}$ with critical numbers at $\frac{1}{e}\Z$.  The
latter filtration is the Moy-Prasad filtration at $x$, where $x\in\B$
is any point in the building lying above the barycenter of the simplex
in $\bB$ corresponding to $P$; the filtration on $U$ comes from a
unique such $x$ which we denote by $x_\cL$.

To see this, fix a $k$-structure for $U$, i.e., a $k$-subspace $V$
such that $U=V\otimes F$.  Recall that any parahoric subgroup in
$\GL(U)$ is conjugate to a \emph{standard parahoric subgroup} with
respect to $V$, i.e., the pullback under $\GL(V\otimes\fo)\to\GL(V)$
of a standard parabolic subgroup $Q\subset \GL(V)$ (so $Q\supset B$).
We may thus assume without loss of generality that $P$ is a standard
parahoric subgroup of this form.  Any lattice chain with stabilizer
contains $V\otimes\fo$, so we may assume that this lattice is $L^0$.
If $e_1,\dots e_n$ is an ordered basis of $V$ compatible with the flag
of $Q$, then there exist indices $s_0=n+1> s_1>s_2>\dots>s_e=1$ such
that for $q\in\Z$ and $0\le j<e$, $L^{qe+j}$ has $\fo$-basis $\{z^q
e_i\mid i<s_j\}\cup\{z^{q+1} e_i\mid i\ge s_j\}$.  Let $\tx=\sum
a_ie^*_{ii}$ where $a_i=\frac{e-j}{e}$ for $s_j\le i<s_{j-1}$, so that
$x\in\Ao$ is the barycenter of the simplex in $\bB$ corresponding to
$P$.  A direct calculation now shows that $\Crit_x=\frac{1}{e}\Z$ and
for all $m\in\Z$, $U_{x,m/e}=L^m$ and $\gl(U)_{x,m/e}=\fP^m$.

In~\cite{BrSa1}, the association of $U$-strata with a formal
flat vector bundle $(U,\n)$ was slightly different than that given here.  We
will show that for strata of positive depth, the two formulations
agree.  For $(U,\n)$ to contain a stratum $(P,s,\b)$ in the sense of
\cite[Definition 4.1]{BrSa1} (with $P$ the stabilizer of a period $e$
lattice chain $\cL$), one first needs $\cL$ to be ``compatible'' with
$\n$.  One can then consider the endomorphism of the associated graded
space $\gr(\cL)=\bigoplus L^i/L^{i+1}$ induced by
$\iota_\tau\circ \n$.  Note that $\b$ also induces
an endomorphism $\gr(\cL)$; indeed, this endomorphism is given by
multiplication by an element $Y\in\fP^{-r}$ corresponding to a
representative $\tb$.  Containment now means that these two
endomorphism coincide on all sufficiently large graded subspaces.  (If
$r>0$, the endomorphisms will actually be the
same.)

By choosing a trivialization of $U$ compatible with $\cL$ (as in
\cite[Remark 2.9]{BrSa1}) and an appropriate $k$-rational structure
$V$ for $U$, one may assume without loss of generality that $P$ is a
standard parahoric subgroup in $\GL_n(\fo)$ associated to the optimal
point $x_\cL\in\Ao$.  For $s>0$, it is now immediate that the
criterion described above is the same as that given in
Proposition~\ref{strataprop} for containment of
$(x_\cL,s/e,\b)$.  Here, we are using the fact that
the normalization term $-\tx\ddz$, which does not appear in
\cite{BrSa1}, makes no contribution in this case.  We are also using
Corollary~\ref{vbgl}, which allows us to apply
Proposition~\ref{strataprop} instead of Definition~\ref{vbcontaindef}.

When $s = 0$, the definitions are not equivalent due to the
normalization term.  However, to
apply results from \cite{BrSa1}, we only need to know that $(U,\n)$
contains a stratum $(x_\cL,0,\b')$ if and only if it contains a
stratum $(P,0,\b)$.  In the latter case, after choosing a
trivialization identifying $L^0\in\cL$ with $\fo^n$, we see that
$\iota_\tau([\n]_\phi)\in\gl_n(\fo)$.  Accordingly,
$\n$ contains the depth $0$ stratum supported at the origin of $\Ao$
determined by $[\n]_\phi$.  Conversely, suppose that $\n$ contains
$(x,0,\b)$ for $x$ an optimal point corresponding to the lattice chain
$\cL$.  Since $[\n]_\phi-\tx\frac{dz}{z}$ and $\tx\frac{dz}{z}$ are
both in $\gl_n(F)^\perp_{x+}$, the same is true for $[\n]_\phi$.
It follows that $\iota_\tau([\n]_\phi)$ preserves every lattice in $\cL$, so by
\cite[Lemme 6.21]{De}, $\n$ is regular singular.  Lemma 4.8 of
\cite{BrSa1} now implies that $\n$ contains a depth $0$ stratum
$(P,0,\b')$.

Summing up the proceeding discussion, we see:
\begin{prop}\label{vbequiv} Let $(U,\n)$ be a flat vector bundle.  Let
  $P\subset\GL(U)$ be a parahoric subgroup corresponding to a lattice
  chain $\cL$ of period $e$, and let $x\in\B$ be any point which has
  the same image in $\bB$ as $x_\cL$.
\begin{enumerate} \item If $s>0$, $(U,\n)$ contains $(P,s,\b)$ in the
  sense of \cite[Definition 4.1]{BrSa1} if and only it contains
  $(x,s/e,\b)$ in the sense of Definition~\ref{vbcontaindef}. 
\item The flat vector bundle $(U,\n)$ contains a depth $0$ stratum
  $(P,0,\b)$ if and only if it contains a depth $0$ stratum
  $(x,0,\b')$.
\end{enumerate}
\end{prop}

\subsection{Slopes of vector bundles via the theory of
  strata}\label{nonbaryproof}

In \cite[Theorem 4.10]{BrSa1}, we showed that the slope of a flat
vector bundle is determined by the strata contained in it associated
to lattice chain filtrations.  We are now ready to prove
Proposition~\ref{nonbary}, which generalizes this result to allow
arbitrary Moy-Prasad filtrations.

\begin{lem} If $(\hV,\n)$ contains a fundamental stratum $(x,r,\b)$
  with respect to the trivialization $\phi$ with
  $x\in\bar{C}\subset\Ao$ and $r>0$, then it also contains a
  fundamental stratum $(x',r,\b')$ with respect to $\phi$ with the
  same depth and $x'\in\bar{C}$ an optimal point.
\end{lem}
\begin{proof} 
  Let $y \in \Ao$ be a vertex adjacent to the alcove containing $x$.
  Since $G = \GL_n$, there exists $n \in \hN$ such that $n y$ is the
  origin and $nx$ is in the fundamental alcove. By Lemma~\ref{actlem}
  part~\eqref{act1}, we see that $[\n]_{n \phi} - \widetilde{n x} \ddz
  \in \Ad^*(n ) ([\n]_{\phi} - \tx \ddz) + \hfg^\vee_{n x, 0+}$.  It
  follows that the stratum $(n x, r, \b')$ determined by $[\n]_{n
    \phi}$ is contained in $(\hV, \n)$.  Moreover, this stratum is
  fundamental if and only if $(x, r, \b)$ is.

  Thus, without loss of generality, we can assume that $C$ is the
  fundamental alcove and $x$ is further in the open star of the
  origin.  In particular, this implies that the corresponding
  filtration on $\hV$ gives rise to a lattice chain $\hV_{x,b_j+qe}$
  with one term given by $V\otimes\fo$, say $\hV_{x,b_0}$; here
  $b_0<b_1<\dots <b_{e-1}<b_0+1$, and the critical numbers of the
  filtration are precisely at the translates of the $b_i$'s.  Up to
  indexing, this is just the lattice chain described above for the
  optimal point in the same open facet.  Accordingly, the graded
  pieces of the filtration are given by $\hV_{x}(b_j+qe)=\spa\{z^{q+1}
  e_j\mid s_j\le j<s_{j-1}\}$.  We claim that $(\hV,\n)$ contains a
  fundamental stratum at an optimal point coming from an appropriate
  sub-lattice chain.

  Let $X\in\gl(\hV)_x(-r)$ be the graded representative corresponding
  to the functional $\tbo$ via the trace form as in
  Proposition~\ref{dual}.  Since $(x,r,\b)$ is fundamental, $X$ is not
  nilpotent.  Hence, for some $j$, $X^m$ is a nonzero map sending
  $\hV_{x}(b_j)$ into $\hV_{x}(b_j-mr)$ for all $m>0$, and so each
  $b_j-mr$ is an integral translate of some $b_i$.  Since there are
  only a finite number of distinct $b_i$'s, $r$ must be a rational
  number with denominator at most $e$, say $r=a/f$ with $(a,f)=1$ and
  $1\le f\le e$.  The set of lattices
  $L^m\overset{\text{def}}{=}\hV_{x}(b_j+m/f)$ for $m\in\Z$ is thus a
  sub-lattice chain of period $f$.  Note that
  $\iota_\tau(([\n]_{\phi})(L^m))\subset L^{m-a}$.  Thus, if $x'$ is
  the optimal point corresponding to this lattice chain, $(\hV,\n)$
  contains the stratum $(x',r,\b')$, with $\b'$ induced by $\n$.
  Finally, we observe that $(\iota_\tau\circ[\n]_{\phi})^m\notin
  \gl(\hV)_{x',-mr+1/f}$ for any $m>0$.  Indeed, if this were true,
  then
  $(\iota_\tau\circ[\n]_{\phi})^m(\hV_{x,b_j})\subset\hV_{x,(b_j-mr)+}$;
  since $r>0$ implies
  $(\iota_\tau\circ[\n]_{\phi})-X\in\gl(\hV)_{x,-r+}$,
  $X^m(\hV_{x,b_j})\subset\hV_{x,(b_j-mr)+}$ as well. This contradicts
  the fact that $(x,r,\b)$ is fundamental.  Applying ~\cite[Lemma
  2.1]{Bus}, we conclude that $(x',r,\b')$ is fundamental.

\end{proof}

\begin{rmk} 
  It is not true that $x'$ can be taken to be an optimal point in the
  same open facet as $x$.  The rank $3$ connection defined in
  Example~\ref{slnexample} contains fundamental strata at $x$ in the
  fundamental alcove $\bar{C}$ for $\SL_3$ precisely for $x$ lying on
  the line connecting the midpoints of two sides of this triangle.

\end{rmk}

\begin{proof}[Proof of Proposition~\ref{nonbary}] By \cite[Theorem
  4.10]{BrSa1}, these statements hold if one only allows $x$ to range
  over uniform filtrations.  By Proposition~\ref{regsing} (for
  $G=\GL_n$), if $(U,\n)$ contains a stratum $(x,0,\b)$, then it also
  contains a stratum $(x',0,b')$ with $x'$ in a minimal facet.  Since
  $x'$ thus corresponds to a uniform filtration, $\slope(U,\n)=0$ and
  $\n$ is regular singular.  We can thus assume that $\n$ is irregular
  singular, i.e., $\slope(\n)>0$.  It now follows from the previous
  lemma that all fundamental strata contained in $(U,\n)$ have depth
  $\slope(\n)$.  It remains to show that one cannot have a
  nonfundamental $(x,r,\b)$ contained in the connection with $0<r\le
  \slope(\n)$.  However, this follows from Corollary~\ref{vbgl} and
  the $\GL_n$ case of Proposition~\ref{minr}.

\end{proof}

We may now combine this proposition with  Corollary~\ref{vbgl} and
Theorem~\ref{MP} to obtain the following
Corollary.

\begin{cor}\label{glslope} The slope of a flat $\GL(V)$-bundle $(\sG, \n)$ is the
  same as the slope of the flat vector bundle $(V_{\sG},\n_V)$.  In
  particular, $(\sG, \n)$ is regular singular if and only if
  $(V_{\sG},\n_V)$ is regular singular.
\end{cor}

\section{Equivariance of stratum containment}\label{equivappendix}

We have seen in Lemma~\ref{strataequivariance} that stratum
containment is well-behaved with respect to change of trivialization.
Here, we show that it also satisfies a stronger equivariance property.
In order to state this, it will be useful to introduce some notation.
Suppose that $x \in \Ao$.  There exists a unique flat structure
$\dAo{x}$ on the trivial $G$ bundle $\hG$ that satisfies the following
properties:
\begin{enumerate}
\item for every $V \in \Rep(G)$, $(\dAo{x} - i \ddz) \hV_{x}(i)  =
  \{0\}$; and
\item for all $g \in \hG$, 
\begin{equation}\label{dA}
(g \dAo{x} g^{-1} - i \ddz) \hV_{g x, i} \subset \Om^1(\hV)_{g x, i+}.
\end{equation}
\end{enumerate}
In fact, we will show that $\dAo{x} = d + \tx \ddz$.  It follows from
Proposition~\ref{ev}\eqref{ev1} that $d + \tx \ddz$ satisfies the
first property.  This immediately shows that the second property holds
when $g$ is the identity.  One obtains the general case of \eqref{dA}
from the observations that $i \ddz = g i \ddz g^{-1}$ and $\hV_{g x,
  i} = g \hV_{x, i}$.

To see that $\dAo{x}$ is unique, note that if $d'$ satisfies the first
condition above, then $d' - \dAo{x}$ is the zero map on all
$\hV_x(i)$.  Since connections are continuous (with respect to the
$z$-adic topology), $d'=\dAo{x}$.

Consider the space $\sQ^G$ of quintuples $(\sG, \n, x, r,\phi)$, where
$(\sG, \n)$ is a formal flat $G$-bundle with trivialization $\phi$,
such that $\n_\phi(\hV_{x, i}) \subset \Om^1(\hV)_{x, i -r}$ for all
$i\in\R$ and $V\in\Rep(G)$. We denote the subset of $\sQ^G$ with fixed
$r$ by $\sQ^G_r$.  The group $\hG$ acts on $\sQ^G$ (and each
$\sQ^G_r$) via $g (\sG, \n, x, r,\phi) = (\sG, \n, g x, r,g\phi)$.

\begin{lem}\label{stratunique} Given $(\sG, \n, x, r,\phi)\in\sQ^G$,
  there is a unique stratum $(x,r,\b)$ contained in $(\sG,\n)$ with
  respect to $\phi$.
\end{lem}
\begin{proof}
  First, we show uniqueness.  Suppose that $(\sG, \n)$ contains both
  $(x, r, \b)$ and $(x, r, \b')$ with respect to $\phi$.  It follows
  that $(X_{\tb} - X_{\tb'}) \hV_{x, i} \subset \Om^1(\hV)_{x, i-r+}$
  for all $i$ and $V \in \Rep(G)$.  In particular, $(X_{\tb} -
  X_{\tb'}) \in \Om^1(\hfg)_{x, -r+}$.  Therefore, $\tb - \tb' \in
  \hfg^\vee_{x, -r+}$ and $\b = \b'$.

  To prove existence, suppose that $x = g y$ for some $g \in \hG$ and
  $y \in \Ao$.  Write $d' = g \dAo{y} g^{-1}$ and $X = \n_\phi - d'$.
  We see that $X$ is $F$-linear; moreover, by \eqref{dA} and the
  assumption on $\n_\phi$, $X \hV_{x, i} \subset \Om^1(\hV)_{x, i-r}$
  for all $V \in \Rep(G)$.  Therefore, $X \in \Om^1(\hfg)_{x, -r}$.
  Since \eqref{dA} implies that $(d' - i\ddz) \hV_{x, i} \subset
  \hV_{x, i+}$, \eqref{stratadefeq} holds with $X_{\tb}$ replaced with
  $X$.  There is a corresponding element $B \in \hfg^\vee_{x, -r}$,
  and we may take $\b$ to be the induced functional on $(\hfg_{x, r} /
  \hfg_{x, r+})$.

\end{proof}

Containment thus induces a map $\sQ^G\to\sS^G$ which restricts to maps
$\sQ^G_r\to\sS^G_r$.  The image of $\sQ^G$ (resp. $\sQ^G_r$) is the
set of strata (resp. strata of depth $r$) contained in some flat
$G$-bundle $(\sG,\n)$ with respect to some
trivialization.

\begin{prop}\label{equivariance}
  The map $\sQ^G\to\sS^G$ and the maps $\sQ^G_r\to\sS^G_r$ are
  $\hG$-equivariant.
\end{prop}
\begin{proof}

  Suppose that $(\sG, \n, x, r,\phi)$ maps to $(x, r, \b)$.
  Lemma~\ref{strataequivariance} implies that $(gx,r,g\b)$ is
  contained in $(\sG,\n)$ with respect to $g\phi$, and
  Lemma~\ref{stratunique} implies that $(gx,r,g\b)$ is the image of
  $(\sG, \n, gx, r,g\phi)$.

\end{proof}

\end{document}